\documentclass{amsart}

\usepackage{mdframed} 
\usepackage[margin=1.5in]{geometry}
\usepackage{latexsym,amsxtra,amscd,ifthen,amsmath,color, amsthm}
\usepackage{amsfonts}

\makeatletter
\@namedef{subjclassname@2020}{%
  \textup{2020} Mathematics Subject Classification}
\makeatother

\usepackage{verbatim}
\usepackage{hyperref}
\usepackage{cancel}
\usepackage{amsthm, tikz, tikz-cd, bbm} 
\usetikzlibrary{arrows}
\usepackage[capitalise]{cleveref}
\usepackage{amssymb}
\usepackage{amsmath}
\usepackage[notcite,notref, final]{showkeys}
\usepackage[all,cmtip]{xy}
\usepackage[normalem]{ulem} 
\usepackage[shortlabels]{enumitem}
\usepackage{xcolor}
\definecolor{darkgreen}{RGB}{55,138,0}
\definecolor{burntorange}{RGB}{180,85,0}
\definecolor{navyblue}{RGB}{18,40,180}
\definecolor{cyan(process)}{rgb}{0.0, 0.6, 1.0}

\hypersetup{
 colorlinks  = true,  
 urlcolor   = blue,  
 linkcolor  = navyblue,  
 citecolor  = black   
}
\allowdisplaybreaks

\usepackage{dcolumn}
\newcolumntype{2}{D{.}{}{5.0}}
\numberwithin{equation}{section}

\theoremstyle{plain}
\newtheorem{theorem}{Theorem}[section]
\newtheorem*{theorem*}{Theorem}
\newtheorem*{definition*}{Definition}
\newtheorem{lemma}[theorem]{Lemma}

\newtheorem{proposition}[theorem]{Proposition}
\newtheorem{corollary}[theorem]{Corollary}
\newtheorem{conjecture}[theorem]{Conjecture}

\theoremstyle{definition}
\newtheorem{definition}[theorem]{Definition}
\newtheorem{notation}[theorem]{Notation}

\newtheorem{example}[theorem]{Example}
\newtheorem{hypothesis}[theorem]{Hypothesis}

\newtheorem{definitiontheorem}[theorem]{Definition-Theorem}
\newtheorem{definitionproposition}[theorem]{Definition-Proposition}

\newtheorem{remark}[theorem]{Remark}
\newtheorem{question}[theorem]{Question}

\newcommand{\stkout}[1]{\ifmmode\text{\sout{\ensuremath{#1}}}\else\sout{#1}\fi}

\newcommand{\kk}{\Bbbk}

\newcommand{\ep}{\varepsilon}

\newcommand{\id}{\textnormal{id}}
\newcommand{\End}{\textnormal{End}}
\newcommand{\RN}[1]{%
  \textup{\uppercase\expandafter{\romannumeral#1}}%
}

\makeatletter              
\let\c@equation\c@theorem  
\makeatother

\begin{document}

\title[On non-counital Frobenius algebras]
{On non-counital Frobenius algebras}


\author{Amanda Hernandez, Chelsea Walton, and Harshit Yadav}

\address{Hernandez: Department of Mathematics, Rice University, Houston, TX 77005, USA}
\email{anh9@rice.edu}

\address{Walton: Department of Mathematics, Rice University, Houston, TX 77005, USA}
\email{notlaw@rice.edu}

\address{Yadav: Department of Mathematics, Rice University, Houston, TX 77005, USA}
\email{hy39@rice.edu}


\begin{abstract} 
A Frobenius algebra is a finite-dimensional algebra $A$ which comes equipped with a coassociative, counital comultiplication map $\Delta$ that is an $A$-bimodule map. Here, we examine comultiplication maps for generalizations of Frobenius algebras: finite-dimensional self-injective (quasi-Frobenius) algebras. We show that large classes  of such algebras, including finite-dimensional weak Hopf algebras, come equipped with a nonzero map $\Delta$ as above  that is not necessarily counital. We also conjecture that this comultiplicative structure holds for self-injective algebras in general.
\end{abstract}


\maketitle




\section{Introduction} \label{sec:intro}

 All algebraic structures in this work are over a field $\kk$ of characteristic 0. Moreover, all algebras are finite-dimensional as $\kk$-vector spaces, and are unital and associative.
This work is motivated by the various characterizations of Frobenius algebras in the literature. These structures were introduced by Frobenius in 1903 in terms of {\it paratrophic matrices} \cite{Frobenius}, and revamped in the 1930s by Brauer-Nesbitt and Nakayama in terms of their representation theory and certain forms that they admit \cite{BN, Nak}; see also \cite[Chapter~6]{Lam}. Starting in the 1990s, the following characterization of a Frobenius algebra became prevalent due to its connection to 2-dimensional Topological Quantum Field Theories (TQFTs). 

\begin{definitiontheorem} \label{defthm:Frob} \cite{Quinn, Abrams}
Let $A$ be an algebra  with multiplication map, $m: A\otimes A\to A$. Then $A$ is {\it Frobenius} if and only if it comes equipped with a $\kk$-linear, coassociative, counital comultiplication map $\Delta: A \to A \otimes A$ that is an $A$-bimodule map,~i.e.,  
\begin{itemize}
    \item   $(\Delta \otimes \id) \Delta = (\id \otimes \Delta) \Delta$ \quad (for coassociativity);\smallskip
    \item $\exists$ $\kk$-linear map $\varepsilon: A \to \kk$ so that 
$(\varepsilon \otimes \id) \Delta = \id = (\id \otimes \varepsilon) \Delta$ \quad (for counitality).
\end{itemize}
Moreover, the left and right $A$-action on $A$ is given by multiplication, and the left (resp., right) $A$-action on $A \otimes A$ is given by left (resp., right) multiplication in the first (resp., second) factor; so, the $A$-bimodule map condition is equivalent to 
\begin{equation} \label{eq:Delta-m}
 (\id \otimes m)(\Delta \otimes \id) = \Delta m =  (m \otimes \id)(\id \otimes \Delta).
\end{equation} 
\end{definitiontheorem}

We examine a similar comultiplicative structure for generalizations of Frobenius algebras: self-injective ({\it quasi-Frobenius}) algebras. We show that large subclasses of self-injective algebras $A$ come equipped with a nonzero comultiplication map $\Delta$ that make $A$ {\it non-counital Frobenius}; that is, $\Delta$ is coassociative, satisfies \eqref{eq:Delta-m}, but is not necessarily counital.

\subsection{Self-injective endomorphism algebras} To proceed, recall an algebra $A$ is called {\it self-injective} if the $A$-modules ${}_A A$ and $A_A$ are both injective. Moreover, an algebra $B$ is {\it basic} if $B_B$ is a direct sum of pairwise non-isomorphic indecomposable (projective) modules.
 In 2006, Skowroński and Yamagata obtained a characterization of self-injective algebras in terms of endomorphism algebras over basic algebras \cite{SYpaper}, given as follows.
\smallskip

Take $B$ a basic  algebra, and let 
$B = P_0 \oplus \cdots \oplus P_{n-1}$ be a decomposition of $B$ into a direct sum of indecomposable right $B$-modules. Here, $n$ is the number of primitive orthogonal idempotents $e_i$ of $B$ for which $1_B\hspace{-.01in} = \hspace{-.01in}\sum_{i=0}^{n-1} e_i$. For  $m_0, \dots, m_{n-1} \hspace{-.015in} \in \hspace{-.02in} \mathbb{Z}_{>0}$, consider the notation:
\begin{equation} \label{eq:B(m_i)}
B(m_0, \dots, m_{n-1}) := \End_B(M_0 \oplus \cdots \oplus M_{n-1}), \quad \text{ for } M_i: = P_i^{\oplus m_i}.
\end{equation}

\begin{theorem} \cite[Theorems~1.3 and 2.1]{SYpaper} \cite[Theorems~IV.6.1,~IV.6.2]{SYbook}
\label{thm:SY}
An algebra $A$ is self-injective if and only if  $A \cong B(m_0, \dots, m_{n-1})$ for $B$ a basic, self-injective algebra. Further, $A$ is Frobenius if and only if $m_i = m_{\nu(i)}$ for all $i$, where $\nu$ is the Nakayama  permutation of $B$. \qed
\end{theorem}

Towards our goal, let us consider an important class of basic algebras $B$ depending on a certain directed graph, or a {\it quiver}, $Q$. Here, we read paths of $Q$ from left-to-right.

\begin{notation} \label{not:Bnl} Take $n \in \mathbb{Z}_{>0}$ and $1 \leq \ell \leq n-1$. Let $Q_{(n)}$ be an $n$-cycle quiver with vertex set $Q_0 = \{0, 1, \dots, n-1\}$ and arrow set $Q_1=\{\alpha_i: i \to i+1\}_{i=0, \dots, n-1}$.  Take $R$ to be the arrow ideal of the path algebra $\kk Q$, and let $\mathcal{I}_\ell$ be the admissible ideal $R^\ell$ of $\kk Q$. Form a  {\it bound quiver algebra} corresponding to the data:
$$B_{n,\ell}:=\kk Q_{(n)} / \mathcal{I}_\ell.$$
\end{notation}

In fact, all  self-injective, {\it connected},  {\it Nakayama} bound quiver  algebras are  of the form $B_{n,\ell}$ \cite[Theorem~IV.6.15]{SYbook}. Therefore, we set the following terminology.

\begin{definition} \label{def:NSY}
Take $A:=B(m_0, \dots, m_{n-1})$, a self-injective algebra from Theorem~\ref{thm:SY}. We refer to $A$ as a {\it Nakayama-Skowro\'{n}ski-Yamagata (NSY) algebra} if $B = B_{n,\ell}$.
\end{definition}

This brings us to our first main result. 

\begin{theorem}[Theorem~\ref{thm:main}] \label{thm:main-intro} The NSY algebras $A :=B(m_0, \dots, m_{n-1})$ are non-counital Frobenius via an explicitly defined nonzero comultiplication map $\Delta: A \to A\otimes A$. In particular, $\Delta$ is counital precisely when $A$ is Frobenius (e.g., when $m_i = m_{\nu(i)}$ for all $i$, as in Theorem~\ref{thm:SY}).
\end{theorem}

We establish Theorem~\ref{thm:main-intro} in Section~\ref{sec:mainresult}. We then provide several examples of comultiplicative structures of NSY algebras in Section~\ref{sec:examples}, such as for the Frobenius algebras $B_{n,\ell}(1,\dots, 1)$, and for Nakayama's 9-dimensional non-Frobenius self-injective algebra $B_{2,2}(2,1)$ \cite[page~624]{Nak}; we also compare the 28-dimensional Frobenius algebra $B_{4,3}(1,2,1,2)$ to the 27-dimensional  algebra $B_{4,3}(1,1,2,2)$ that is not counital.

\subsection{Finite-dimensional weak Hopf algebras} Next, we turn our attention to another important class of self-injective algebras: finite-dimensional weak Hopf algebras. A {\it weak Hopf algebra} is an algebra $H$ that is equipped with a $\kk$-linear coassociative comultiplication map $\Delta_{\textnormal{wk}}$, which is counital and satisfies certain compatibility axioms different than \eqref{eq:Delta-m}; see Definition~\ref{def:weak}. Such algebras are generalizations of Hopf algebras, which have gained prominence due to extensive work by Böhm-Nill-Szlanchányi in the late 1990s (see, e.g., \cite{BNS}). One of their main results is the following.

\begin{theorem} \label{thm:BNS-intro} \cite[Theorems~3.11 and~3.16]{BNS}
Finite-dimensional weak Hopf algebras $H$ are self-injective. Further, $H$ is Frobenius if and only if $H$ has a non-degenerate left integral (see Definition~\ref{def:integral}). \qed
\end{theorem}

See \cite{IK} and \cite[Section~9.2]{HH} for a study of finite-dimensional weak Hopf algebras that are not counital Frobenius.
Now we fulfill our goal for finite-dimensional weak Hopf algebras as described below.

\begin{theorem}[Theorem~\ref{thm:main-weak}] \label{thm:main-weak-intro} Finite-dimensional weak Hopf algebras $H$ are non-counital Frobenius  via an explicitly defined nonzero comultiplication map $\Delta: H \to H\otimes H$. In particular, $\Delta$ is counital precisely when $H$ is Frobenius (e.g., when $H$ has a non-degenerate left integral). 
\end{theorem}

Background material on weak Hopf algebras  and the proof of Theorem~\ref{thm:main-weak-intro} can be found in Section~\ref{sec:weak}. We also provide many examples of this result in Section~\ref{sec:weak-ex}, especially for {\it groupoid algebras}, and for {\it quantum transformation groupoids}.

\subsection{Further directions} Our two results above, Theorems~\ref{thm:main-intro} and~\ref{thm:main-weak-intro}, prompt the following conjecture, progress of which is illustrated in Diagram~1 below.

\begin{conjecture} \label{main-conj}
Self-injective algebras are non-counital Frobenius. In particular, given a class of self-injective algebras, there exists a nonzero comultiplication map $\Delta$ for the members of the class yielding  non-counital Frobenius structures, such that $\Delta$ is counital precisely when a member is Frobenius.
\end{conjecture}

\vspace{.1in}

\begin{center}

\scalebox{0.8}{

\tikzset{every picture/.style={line width=0.75pt}} 

\begin{tikzpicture}[x=0.75pt,y=0.75pt,yscale=-1,xscale=1]

\draw   (3.34,100.01) .. controls (3.34,47.92) and (131.68,5.68) .. (290,5.68) .. controls (448.32,5.68) and (576.67,47.92) .. (576.67,100.01) .. controls (576.67,152.11) and (448.32,194.35) .. (290,194.35) .. controls (131.68,194.35) and (3.34,152.11) .. (3.34,100.01) -- cycle ;
\draw   (149.34,112.01) .. controls (149.34,73.54) and (213.06,42.35) .. (291.67,42.35) .. controls (370.28,42.35) and (434,73.54) .. (434,112.01) .. controls (434,150.49) and (370.28,181.68) .. (291.67,181.68) .. controls (213.06,181.68) and (149.34,150.49) .. (149.34,112.01) -- cycle ;
\draw   (35.34,124) .. controls (35.34,98.42) and (97.42,77.68) .. (174,77.68) .. controls (250.59,77.68) and (312.67,98.42) .. (312.67,124) .. controls (312.67,149.58) and (250.59,170.32) .. (174,170.32) .. controls (97.42,170.32) and (35.34,149.58) .. (35.34,124) -- cycle ;
\draw   (270,123.33) .. controls (270,97.75) and (332.09,77.01) .. (408.67,77.01) .. controls (485.25,77.01) and (547.34,97.75) .. (547.34,123.33) .. controls (547.34,148.91) and (485.25,169.65) .. (408.67,169.65) .. controls (332.09,169.65) and (270,148.91) .. (270,123.33) -- cycle ;

\draw (230,57) node [anchor=north west][inner sep=0.75pt]  [font=\Large] [align=left] {{\fontfamily{ptm}\selectfont Frobenius algebras}};
\draw (218,17) node [anchor=north west][inner sep=0.75pt]  [font=\Large] [align=left] {{\fontfamily{ptm}\selectfont self-injective algebras}};
\draw (317,100) node [anchor=north west][inner sep=0.75pt]  [font=\normalsize] [align=left] {\begin{minipage}[lt]{75pt}\setlength\topsep{0pt}
\begin{center}
{\fontfamily{ptm}\selectfont {\normalsize  \ \ with ... }}\\{\fontfamily{ptm}\selectfont {\normalsize non-degenerate}}\\{\fontfamily{ptm}\selectfont {\normalsize left integral}}
\end{center}

\end{minipage}};
\draw (428,100) node [anchor=north west][inner sep=0.75pt]  [font=\normalsize] [align=left] {\begin{minipage}[lt]{76.13pt}\setlength\topsep{0pt}
\begin{center}
{\fontfamily{ptm}\selectfont {\normalsize ... finite-dim'l \ \ \ \ }}\\{\fontfamily{ptm}\selectfont {\normalsize weak Hopf algebras}}
\end{center}

\end{minipage}};
\draw (55,100) node [anchor=north west][inner sep=0.75pt]  [font=\normalsize] [align=left] {\begin{minipage}[lt]{68pt}\setlength\topsep{0pt}
\begin{center}
{\fontfamily{ptm}\selectfont {\normalsize NSY algebras ... }}
\end{center}

\end{minipage}};
\draw (48,120) node [anchor=north west][inner sep=0.75pt]  [font=\normalsize]  {$B( m_{0} ,...,m_{n}{}_{-1})$};
\draw (145,98) node [anchor=north west][inner sep=0.75pt]  [font=\normalsize]  {$ \begin{array}{rl}
...\; m_{i} \ &=m_{\nu }{}_{( i)}, \ \\
\nu ( i) &=\ i+\ell -1\ \\
\ \ \ \ \ \ &\ \ \ \ \ \ \bmod \ n\\
\end{array}$};

\end{tikzpicture}
}

\vspace{.15in}
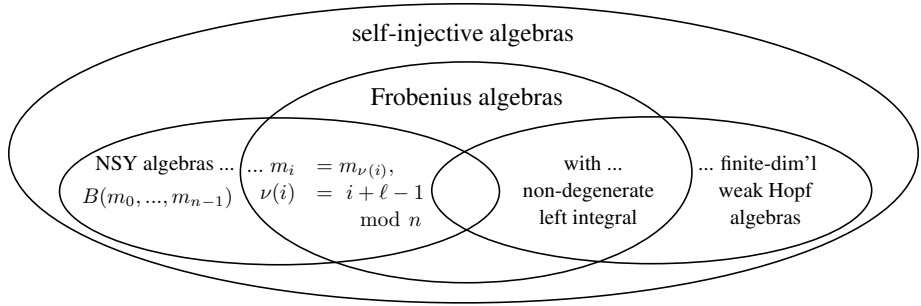
{\small Diagram 1:  Conjecture~\ref{main-conj} holds for NSY-algebras [Theorem~\ref{thm:main-intro}]\\ and for finite-dimensional weak Hopf algebras [Theorem~\ref{thm:main-weak-intro}].} 
\end{center}

\vspace{.15in}

The conjecture is likely to hold because the techniques used to resolve the conjecture for the NSY algebras could be close to its resolution in the general case. Indeed, the bound quiver algebras $B_{n,\ell}$ form a  prototypical class of basic self-injective algebras $B$.

\smallskip

In fact, the conclusion of Conjecture~\ref{main-conj} has been achieved for another generalization of Frobenius  algebras: gendo-Frobenius algebras,  introduced recently by Yırtıci. See \cite[Theorem~4.3]{Yir}. 
Here, an algebra is said to be {\it Morita} if it is isomorphic to the endomorphism algebra of a finite-dimensional, faithful (right) module $M$ over a self-injective algebra $B$  \cite{KerYam}. Further, an algebra is called {\it gendo-Frobenius} if it is Morita under the conditions that $B$ is Frobenius and that $M \cong M_{\nu_B}$ as right $B$-modules, where $\nu_B$ is the {\it Nakayama automorphism} of $B$. 
With this observation, we end with a few directions for further research.

\begin{question}
Does the conclusion of Conjecture~\ref{main-conj} hold for Morita algebras?
\end{question}

\begin{question}
What are the intersections between the classes of algebras above for which Conjecture~\ref{main-conj} holds? For instance, what are examples of {\it NSY weak Hopf algebras}?
\end{question}

\begin{question}
What are the physical uses of non-counital Frobenius algebras, say, akin to the use of (counital) Frobenius algebras in 2-dimensional TQFTs?
\end{question}

See, e.g., work of Cohen and Godin \cite{CohenGodin} for connections between non-counital Frobenius algebras and  {\it  positive
boundary 2-dimensional TQFTs}.


\section{NSY algebras are non-counital Frobenius}  \label{sec:mainresult}
In this section, we establish an explicit non-counital Frobenius structure for the NSY algebras [Definition~\ref{def:NSY}], thereby proving Conjecture~\ref{main-conj} for a large class of self-injective algebras. Notation and preliminary results are provided in Section~\ref{sec:prelim}; the basis of the NSY algebras is discussed in Section~\ref{sec:basis} (see Proposition~\ref{prop:Bmi-basis}); and the unital, multiplicative structure of the NSY algebras is provided in Section~\ref{sec:algebra} (see Proposition~\ref{prop:Bmi-alg}). The main result on the non-counital Frobenius structure of the NSY algebras is then presented in Section~\ref{sec:Frobenius}, incorporating the results above (see Theorem~\ref{thm:main}).

 \subsection{Preliminaries} \label{sec:prelim}  To proceed, recall the notation in \eqref{eq:B(m_i)} and Notation~\ref{not:Bnl} for the NSY algebras $B_{n,\ell}(m_0, \dots, m_{n-1})$  in Definition~\ref{def:NSY}.

\begin{notation} \label{not:X}
Let $e_i$ denote the trivial path at vertex $i$ of $Q_0$. Consider the decomposition of $B_{n, \ell}$ into a direct sum of indecomposable right $B_{n, \ell}$-modules, $\bigoplus_{i=0}^{n-1} P_i$, for $P_i := e_i B_{n,\ell}$. 
\begin{itemize}
    \item Let the basis of $B_{n,\ell}$ be paths of $Q$ starting at vertex $i$ of length $k$, denoted by
$$\alpha_{i,k}:= \alpha_i \; \alpha_{i+1} \; \cdots \; \alpha_{i+k-1}  \quad \text{ for } 
0 \leq i \leq n-1, \; 0 \leq k \leq \ell-1.$$ 
Here, $\alpha_{i,0} = e_i$. In particular, the basis of $P_i$ is given by $\{\alpha_{i,k}\}_{0 \leq k \leq \ell-1}$.
\medskip
\item Let $P_i^{r_i}$ be the $r_i$-th copy of $P_i$, for $0 \leq r_i \leq m_i - 1$, and denote its basis by $\{\alpha^{r_i}_{i,k}\}_{0 \leq k \leq \ell-1}$. Moreover, $P_i^{r_i}$ is a right $B_{n,\ell}$-module via
\begin{equation} \label{eq:P-act}
\alpha^{r_i}_{i,k} \cdot \beta = (\alpha_{i,k} \cdot \beta)^{r_i}, \quad \text{for}  \quad \beta \in B_{n,\ell}.
\end{equation}

\smallskip

\item Consider the maps
\begin{equation} \label{eq:Xij}
X_{i,j}^{r_i, s_{i+j}}: P_{i+j}^{s_{i+j}} \longrightarrow P_i^{r_i}, \quad  \alpha^{s_{i+j}}_{i+j,k} \mapsto  (\alpha_{i,j} \cdot \alpha_{i+j,k})^{r_i} = \alpha_{i, j+k}^{r_i}.
\end{equation}
That is, $X_{i,j}^{r_i, s_{i+j}}$ is pre-composition by the path $\alpha_{i,j}$ (in the appropriate copy of $P_i$). 

\medskip

\item Extend $X_{i,j}^{r_i, s_{i+j}}$ to an endomorphism of $\bigoplus_{i,r_i} P_i^{r_i}$ by setting 
$$X_{i,j}^{r_i,s_{i+j}}(P_a^{s_a}) = 0, \quad \text{for \;$a \neq i+j, \; \;s_a \neq s_{i+j}$.}$$

\medskip

\item Indices are here as follows: $n,\; \ell,\; m_0, \dots, m_{n-1}$ are fixed; $\; i, \; a$ are taken modulo $n$; and $0 \leq j,k,b \leq \ell-1, \; \;   0 \leq r_i \leq m_i -1, \; \;  0 \leq s_{i+j} \leq m_{i+j} -1.$ 
\end{itemize}
\end{notation}

Next, we compute the  Nakayama permutation \cite[page~377]{SYbook} \cite[page~136]{SYpaper} of the NSY algebras, as this is needed to determine  when such algebras are  Frobenius [Theorem~\ref{thm:SY}]. We refer to \cite[Section~I.3]{ASSbook} for background on radicals, tops, and socles of modules of finite-dimensional algebras.

 \begin{proposition} \label{prop:Nak-perm}
 The Nakayama permutation $\nu$ of $B_{n, \ell}$ is the permutation of $\{1, \dots, n\}$ given by 
 $\nu(i) = i + \ell -1 \text{ modulo $n$}.$
 \end{proposition}
 
 \begin{proof}
 Recall Notation~\ref{not:X}; in particular, the indecomposable right $B_{n, \ell}$-modules are of the form $P_i  = e_i B_{n,\ell} =  \bigoplus_{k=0}^{\ell-1} \kk \alpha_{i,k}$, with $\alpha_{i,0} = e_i$. By the definition on  \cite[page 136]{SYpaper}, we need to compute the permutation of $\nu$ of $\{1, \dots, n\}$ such that 
\[
\text{soc}(e_i B_{n,\ell}) \; \cong \; \text{top}(e_{\nu(i)} B_{n,\ell}).
\]
 Now rad($P_i$)=$\bigoplus_{k=1}^{\ell-1} \kk \alpha_{i,k}$. So,  we obtain that top($P_i$) = $P_i/\text{rad}(P_i)$ = $\kk e_i$. On the other hand, soc($P_i$) = $\kk \alpha_{i, \ell-1} \cong \kk e_{i+\ell -1}$ as right $B_{n,\ell}$-modules. Here, $\kk e_{i+\ell -1}$ is the right  $B_{n, \ell}$-module with action $ e_{i+\ell -1} \cdot \alpha_{j,k} =\delta_{i+\ell-1, j} \;\delta_{k,0}\; e_{i + \ell -1}$. So, $\nu(i) = i + \ell -1$ modulo $n$. 
 \end{proof}


 \subsection{Basis} \label{sec:basis}
Next, we show that the maps $\{X_{i,j}^{r_i,s_{i+j}}\}$ from Notation~\ref{not:X} form a basis of a given NSY algebra $B_{n,\ell}(m_0,\dots,m_{n-1})$. Consider the following preliminary results.

\begin{lemma}
We have that $X_{i,j}^{r_i,s_{i+j}}\in B_{n,\ell}(m_0, \dots, m_{n-1}) $.
\end{lemma}

\begin{proof}
 We need to show that $X_{i,j}^{r_i, s_{i+j}}$ is a map of right $B_{n,\ell}$-modules. Take $\beta \in B_{n,\ell}$, and take $\theta \in  P_a^{s_a}$ for some $0 \leq a \leq n-1$ and $0 \leq s_a \leq m_a -1$. Then, $\theta = \gamma^{s_a}$, for some path $\gamma$ that starts at vertex $a$. Now we see that $X_{i,j}^{r_i, s_{i+j}}$ is a right $B_{n,\ell}$-module map as follows:
\[
{\small
\begin{array}{rll}
\medskip
X_{i,j}^{r_i, s_{i+j}}(\theta \cdot \beta) 
&\; = \; X_{i,j}^{r_i, s_{i+j}}(\gamma^{s_a} \cdot \beta) 
&\; \overset{\textnormal{\eqref{eq:P-act}}}{=} \; X_{i,j}^{r_i, s_{i+j}}(\gamma \cdot \beta)^{s_a}\\
\medskip
&\overset{\textnormal{\eqref{eq:Xij}}}{=} \; \delta_{i+j,a}\;  \delta_{s_{i+j},s_a}\; (\alpha_{i,j} \cdot \gamma \cdot \beta)^{r_i} 
&\; \overset{\textnormal{\eqref{eq:P-act}}}{=} \; \delta_{i+j,b}\;  \delta_{s_{i+j},s_a}\; [(\alpha_{i,j} \cdot \gamma)^{r_i} \cdot \beta]\\
\medskip
&\; = \; [\delta_{i+j,a}\;  \delta_{s_{i+j},s_a}\; (\alpha_{i,j} \cdot \gamma)^{r_i}] \cdot \beta
&\; = X_{i,j}^{r_i, s_{i+j}}(\gamma^{s_a}) \cdot \beta \\
&\; = \; X_{i,j}^{r_i, s_{i+j}}(\theta) \cdot \beta.
\end{array}
}
\]

\vspace{-.18in}

\end{proof}

\begin{lemma}\label{lem:Bmi-linind}
The elements $X_{i,j}^{r_i,s_{i+j}}$ are linearly independent.
\end{lemma}
\begin{proof}
Suppose not, then there exist scalars 
$ c_{i,j}^{r_i,s_{i+j}} \in \kk$ not all zero such that 
\begin{equation}\label{eq:li1}
    \sum\limits_{i=0}^{n-1} \; \sum\limits_{j=0}^{\ell-1}\; \sum\limits_{r_i=0}^{m_i-1}\; \sum\limits_{s_{i+j}=0}^{m_{i+j}-1} c_{i,j}^{r_i,s_{i+j}} X_{i,j}^{r_i,s_{i+j}} = 0. 
\end{equation}
Since $ c_{i,j}^{r_i,s_{i+j}}$ are not all zero, there exists indices $a$, $b$, $r_a$, $s_{a+b}$ such that $ c_{a,b}^{r_a,s_{a+b}}\neq 0$. Now, evaluate~\eqref{eq:li1} on the element $\alpha_{a+b,0}^{s_{a+b}}$ to get that
{\small
\begin{align*}
    & \sum\limits_{i=0}^{n-1}\; \sum\limits_{j=0}^{\ell-1}\; \sum\limits_{r_i=0}^{m_i-1}\; \sum\limits_{s_{i+j}=0}^{m_{i+j}-1} c_{i,j}^{r_i,s_{i+j}} \; X_{i,j}^{r_i,s_{i+j}} (\alpha_{a+b,0}^{s_{a+b}})\\
    & = 
    \sum\limits_{i=0}^{n-1}\; \sum\limits_{j=0}^{\ell-1}\; \sum\limits_{r_i=0}^{m_i-1}\; \sum\limits_{s_{i+j}=0}^{m_{i+j}-1} c_{i,j}^{r_i,s_{i+j}} \; \delta_{i+j,a+b}\; \delta_{s_{i+j},s_{a+b}}  \alpha_{i,j}^{r_i}  \\
    &= \sum\limits_{i=0}^{n-1} \; \sum\limits_{j=0}^{\ell-1} \;  \sum\limits_{r_i=0}^{m_i-1} c_{i,j}^{r_i,s_{a+b}} \; \delta_{i+j,a+b} \; \alpha_{i,j}^{r_i}   \\
    & = 0.
\end{align*}
}

\noindent Since $\{ \alpha_{i,j}^{r_i} \}$ are linearly independent elements of $\bigoplus_{i,r_i} P_i^{r_i}$, we get that $c_{i,j}^{r_i, s_{a+b}} = 0$ for all $i$, $j$, $r_i$. In particular, for $i = a$ and $j=b$, we get that $c_{a,b}^{r_a,s_{a+b}}=0$. This is a contradiction to our assumption. Thus, the claim follows.
\end{proof}

\begin{lemma}\label{lem:Bmi-maps}
 Every nonzero right $B_{n,\ell}$-module map $\phi$ from $P_a^{r_a}$ to $P_i^{r_i}$ is of the form $\phi(\alpha_{a,k}^{r_a}) = (\alpha_{i,j}^{r_i})\cdot \alpha_{a,k}$ for some $j$ such that $i+j\equiv a\textnormal{ mod }n$.
\end{lemma}

\begin{proof}
 Since $\phi$ is a right $B_{n,\ell}$-module map, we have that
\[ 0 \neq \phi(\alpha_{a,0}^{r_a}) \overset{\textnormal{\eqref{eq:P-act}}}{=} \phi( \alpha_{a,0}^{r_a}\cdot \alpha_{a,0} ) = \phi(\alpha_{a,0}^{r_a})\cdot \alpha_{a,0}. \] 
Therefore, $\phi(\alpha_{a,0}^{r_a})$ ends at vertex $a$. Also, $\phi(\alpha_{a,0}^{r_a})\in P_i^{r_i}$ implies that $\phi(\alpha_{a,0}^{r_a})$ starts at vertex~$i$. Thus, it is clear that $\phi(\alpha_{a,0}^{r_a}) = \alpha_{i,j}^{r_i}$ for some $j$ such that $i+j\equiv a \textnormal{ mod }n$. Hence, the claim holds.
\end{proof}

\begin{lemma}  \label{lem:Bmi-dim}
We have that $\dim_\kk \left(B_{n,\ell}(m_0, \dots, m_{n-1}) \right) = \sum_{i=0}^{n-1} \sum_{j=0}^{\ell -1} m_i m_{i+j}.$ 
\end{lemma}

\begin{proof}
 Recall that 
{\small
\begin{align*}
B_{n,\ell}(m_0,\ldots, m_{n-1}) & = \textnormal{Hom}_{B_{n,\ell}}\left(\bigoplus\limits_{a=0}^{n-1}\; \bigoplus\limits_{r_a=0}^{m_a-1} P_a^{r_a}, \; \;  \bigoplus\limits_{i=0}^{n-1} \; \bigoplus\limits_{r_i=0}^{m_i-1} P_i^{r_i} \right) \\
& \cong \bigoplus\limits_{a,i=0}^{n-1} \; \bigoplus\limits_{r_a=0}^{m_a-1} \; \bigoplus\limits_{r_i=0}^{m_i-1} \textnormal{Hom}_{B_{n,\ell}}(P_a^{r_a},P_i^{r_i}).
\end{align*} 
}

\noindent By Lemma~\ref{lem:Bmi-maps}, dim$_\kk(\textnormal{Hom}_{B_{n,\ell}}(P_a^{r_a},P_i^{r_i})) = \#\{ 1\leq j\leq \ell-1 ~:~ i+j\equiv a \textnormal{ mod }n  \}$. 
Therefore, 
{\small
\begin{align*}
    \textnormal{dim}_\kk(B_{n,\ell}(m_0,\ldots, m_{n-1})) & = \sum\limits_{a=0}^{n-1} \; \sum\limits_{i=0}^{n-1} \; \sum\limits_{r_a=0}^{m_a-1}\; \sum\limits_{r_i=0}^{m_i-1}  \#\{ 1\leq j\leq \ell-1 ~:~ i+j\equiv a \textnormal{ mod }n  \} \\
    & = \sum\limits_{a=0}^{n-1}\; \sum\limits_{i=0}^{n-1} m_a m_i \; \#\{ 1\leq j\leq \ell-1 ~:~ i+j\equiv a \textnormal{ mod }n  \} \\
    & = \sum\limits_{i=0}^{n-1} m_i \; \sum\limits_{a=0}^{n-1} m_a \; \# \{ 1\leq j\leq \ell-1 ~:~ i+j\equiv a \textnormal{ mod }n  \} \\
    & = \sum_{i=0}^{n-1} m_i \; \sum_{j=0}^{\ell -1} m_{i+j}.
\end{align*}
}

\vspace{-.2in} 

\end{proof}

\begin{proposition} \label{prop:Bmi-basis}
The collection of morphisms $\{X_{i,j}^{r_i, s_{i+j}}\}_{i,j,r_i,s_{i+j}}$ form a $\kk$-basis of   \linebreak $B_{n,\ell}(m_0, \dots, m_{n-1})$. 
\end{proposition}

\begin{proof}
The number of elements in the linearly independent set $\{X_{i,j}^{r_i, s_{i+j}}\}_{i,j,r_i,s_{i+j}}$ from Lemma~\ref{lem:Bmi-linind} is $\sum_{i=0}^{n-1} \sum_{j=0}^{\ell -1} m_i m_{i+j}$, the same as dim$_\kk (B_{n,\ell}(m_0,\ldots,m_{n-1}))$ by Lemma~\ref{lem:Bmi-dim}. So, the result follows.
\end{proof}


 \subsection{Algebra structure} \label{sec:algebra}

Next, we provide an explicit algebraic structure of the NSY algebras $A:=B_{n,\ell}(m_0, \dots, m_{n-1})$ in terms of the basis of $A$ from Proposition~\ref{prop:Bmi-basis}. In particular, we know that such $A$ are associative and unital algebras, but we seek the explicit product and unit formulas in terms of the basis $\{X_{i,j}^{r_i, s_{i+j}}\}_{i,j,r_i,s_{i+j}}$ in order to derive the desired coproduct formula for $A$ later (towards proving Theorem~\ref{thm:main-intro}).

\begin{proposition} \label{prop:Bmi-alg}
Take $A:=B_{n,\ell}(m_0, \dots, m_{n-1})$, an NSY algebra. Then, with the elements $\{X_{i,j}^{r_i, s_{i+j}}\}_{i,j,r_i,s_{i+j}}$ from Notation~\ref{not:X}, the formulas 
\[
\begin{array}{c}
\medskip
X_{i,j}^{r_i, s_{i+j}} \cdot  X_{a,b}^{r_a, s_{a+b}}  = \begin{cases}
\delta_{a,i+j} \; \delta_{r_a, s_{i+j}}\;  X_{i,j+b}^{r_i, s_{a+b}}, &\text{for $j +b < \ell$},\\ 
0, &\text{else};
\end{cases}
\medskip \\
1_A = \sum_ {i=0}^{n-1} \sum_{r_i=0}^{m_{i}-1} X_{i,0}^{r_i,r_i},
\end{array}
\]
give $A$  an explicit structure of an associative, unital algebra.
\end{proposition}

\begin{proof}
First, we study the multiplication of $A$ in terms of the basis $\{X_{i,j}^{r_i, s_{i+j}}\}_{i,j,r_i,s_{i+j}}$. Given the two basis elements $X_{i,j}^{r_i, s_{i+j}}$ and $X_{a,b}^{r_a, s_{a+b}}$ as above, their composition is defined to be:
$$X_{i,j}^{r_i, s_{i+j}} \circ X_{a,b}^{r_a, s_{a+b}} = \begin{cases} 
      P_{a+b}^{s_{a+b}} \longrightarrow P_a^{r_a} = P_{i+j}^{s_{i+j}} \longrightarrow P_i^{r_i}, & a = i+j ,\; r_a = s_{i+j} , \; j+b < \ell \\
      0, & \text{otherwise}. 
   \end{cases} $$
If we consider a basis element $\alpha_{a+b, k}^{s_{a+b}}$ of $P_{a+b}^{s_{a+b}}$, then when $a = i+j , \; r_a = s_{i+j} ,\;  j+b < \ell$, this map is defined by:
\[ 
{\small
\begin{array}{rll}
\medskip
X_{i,j}^{r_i, s_{i+j}} \circ X_{a,b}^{r_a, s_{a+b}} \left(\alpha_{a+b, k}^{s_{a+b}} \right) &\; = \;  X_{i,j}^{r_i, s_{i+j}} \left(\left(\alpha_{a,b} \cdot \alpha_{a+b, k} \right)^{r_a} \right) 
  &\; = \; \left( \alpha_{i,j} \cdot \alpha_{a,b} \cdot \alpha_{a+b,k} \right)^{r_i} \\ \medskip
   &\; = \; \left( \alpha_{i,j} \cdot \alpha_{i+j,b} \cdot \alpha_{a+b,k} \right)^{r_i} 
   &\; = \; \left( \alpha_{i,j+b} \cdot \alpha_{a+b,k} \right)^{r_i} \\
   &\; = \; X_{i,j+b}^{r_i, s_{a+b}} \left( \alpha_{a+b,k}^{s_{a+b}} \right).
\end{array} 
}
\]
Here, we use~\eqref{eq:P-act} throughout. Hence, we have the product formula for $A$ as claimed. 

\smallskip

Though not necessary, we show next that the product formula yields an associative multiplication:  
{\small
\begin{align*}
\medskip
&\left(X_{i,j}^{r_i, s_{i+j}} \cdot  X_{a,b}^{r_a, s_{a+b}} \right) \cdot X_{c,d}^{r_c,s_{c+d}} \\
&\qquad = 
\begin{cases}
\delta_{a,i+j} \; \delta_{r_a, s_{i+j}}\;  X_{i,j+b}^{r_i, s_{a+b}} \cdot X_{c,d}^{r_c,s_{c+d}}, &\text{for $j +b < \ell$}\\
0, &\text{else}
\end{cases} \bigskip \\ 
&\qquad = 
\begin{cases}
\delta_{a,i+j} \; \delta_{r_a, s_{i+j}}\; \delta_{c,i+j+b} \; \delta_{r_c,s_{i+j+b}}\; X_{i,j+b+d}^{r_i, s_{c+d}}, &\text{for $j+b+d < \ell$},\\
0, & \text{else}
\end{cases} \bigskip \\
&\qquad = 
\begin{cases}
\delta_{c,a+b} \; \delta_{r_c,s_{a+b}} \; \delta_{a,i+j} \; \delta_{r_a, s_{i+j}}\; X_{i,j+b+d}^{r_i, s_{c+d}}, &\text{for $j+b+d < \ell$},\\
0, & \text{else}
\end{cases} \bigskip \\
&\qquad = 
\begin{cases}
\delta_{c,a+b} \; \delta_{r_c,s_{a+b}}\;  X_{i,j}^{r_i, s_{i+j}} \cdot X_{a,b+d}^{r_a, s_{c+d}}, &\text{for $b+d < \ell$},\\
0, & \text{else}
\end{cases} \bigskip\\
&\qquad = X_{i,j}^{r_i, s_{i+j}} \cdot \left( X_{a,b}^{r_a, s_{a+b}} \cdot X_{c,d}^{r_c,s_{c+d}} \right).
\end{align*}
}

We now verify the unitality axiom, with  $1_A$ given in terms of the basis $\{X_{i,j}^{r_i, s_{i+j}}\}_{i,j,r_i,s_{i+j}}$ as claimed. Consider the  computations below:
\[ 
{\small
\begin{array}{rll}
\medskip
1_A \cdot X_{a,b}^{r_a, s_{a+b}} &= \left( \sum_ {i=0}^{n-1} \sum_{r_i=0}^{m_{i}-1} X_{i,0}^{r_i,r_i} \right) \cdot  X_{a,b}^{r_a, s_{a+b}} & = \sum_ {i=0}^{n-1} \sum_{r_i=0}^{m_{i}-1} \left(  X_{i,0}^{r_i,r_i} \cdot  X_{a,b}^{r_a, s_{a+b}} \right) \\ \medskip
& = \sum_ {i=0}^{n-1} \sum_{r_i=0}^{m_{i}-1} \left(\delta_{a,i} \; \delta_{r_a,r_i} \; X_{i,b}^{r_i, s_{a+b}} \right) & =  \sum_ {i=0}^{n-1} \left(\delta_{a,i} \; X_{i,b}^{r_a, s_{a+b}} \right) \\
\bigskip
& =    X_{a,b}^{r_a, s_{a+b}}, \\ \medskip
X_{a,b}^{r_a, s_{a+b}} \cdot 1_A &\; = \;   X_{a,b}^{r_a, s_{a+b}} \cdot  \left( \sum_ {i=0}^{n-1} \sum_{r_i=0}^{m_{i}-1} X_{i,0}^{r_i,r_i} \right) 
&= \sum_ {i=0}^{n-1} \sum_{r_i=0}^{m_{i}-1} \left(X_{a,b}^{r_a, s_{a+b}} \cdot X_{i,0}^{r_i,r_i}
\right) \\ \medskip
&= \sum_ {i=0}^{n-1} \sum_{r_i=0}^{m_{i}-1} \delta_{i,a+b}\; \delta_{r_i, s_{a+b}} X_ {a,b}^{r_a,r_i} & =   \sum_{r_i=0}^{m_{i}-1} \delta_{r_i, s_{a+b}} X_ {a,b}^{r_a,r_i} \\
& =   X_ {a,b}^{r_a,s_a+b}.
\end{array}
}
\]
Therefore, the result holds.
\end{proof}


 \subsection{Non-counital Frobenius structure} \label{sec:Frobenius}
Here, we establish a non-counital Frobenius structure for the NSY algebras, building on the algebra structure presented in Proposition~\ref{prop:Bmi-alg}. Consider the next preliminary result.

\begin{lemma} \label{lem:coassoc} 
Let $A$ be an algebra with $1_A$, and define a $\kk$-linear map
$\Delta: A \to A \otimes A$ by $\Delta(1_A):= \textstyle \sum_i a_i \otimes b_i$ and $\Delta(x) :=  \sum_i a_i \otimes b_i x.$
If 
\begin{equation}
\label{eq:Delta1}
\textstyle \sum_i a_i \otimes b_i x = \sum_i x a_i \otimes b_i, \quad \quad \text{ for all } x \in A,
\end{equation} 
then $(A,m,u,\Delta)$ is a non-counital Frobenius algebra.
\end{lemma}

\begin{proof}
It suffices to show that $\Delta$ is coassociative and satisfies \eqref{eq:Delta-m}.  Now $\Delta$ is coassociative as, for all $y \in A$, we get that
\begin{align*}
\textstyle (\Delta \otimes \id)\Delta(y) 
& \; \; =  \; \; \textstyle \sum_i \Delta(a_i) \otimes b_i y 
 \; \; =  \; \; \textstyle \sum_{i,j} a_j \otimes b_j a_i \otimes b_i y  \; \; \overset{\textnormal{\eqref{eq:Delta1}, $x \hspace{-.03in}=\hspace{-.03in}y$}}{=}  \; \; \sum_{i,j} a_j \otimes b_j y a_i \otimes b_i\\
&\overset{\textnormal{\eqref{eq:Delta1}, $x\hspace{-.03in}=\hspace{-.03in}b_j y$}}{=}  \; \; \textstyle \sum_{j,i} a_j \otimes a_i \otimes b_i b_j y  \; \; =  \; \; \sum_j a_j \otimes \Delta(b_j y)  \; \; =  \; \; (\id \otimes \Delta)\Delta(y). 
\end{align*}
Moreover, \eqref{eq:Delta-m} holds as for all $y,z \in A$, we get that
\begin{align*}
(\id \otimes m)(\Delta \otimes \id)(y \otimes z) = (\id \otimes m)(\textstyle \sum_i a_i \otimes b_i y \otimes z) =  \textstyle \sum_i a_i \otimes b_i y z = \Delta m (y \otimes z)\\
(m \otimes \id)(\id \otimes \Delta)(y \otimes z)  =  \textstyle \sum_i y  a_i \otimes b_i z 
\overset{\textnormal{\eqref{eq:Delta1},  $x\hspace{-.03in}=\hspace{-.03in}y$}}{=}  \textstyle \sum_i a_i \otimes b_i y z = \Delta m (y \otimes z).
\end{align*} 
Thus, $(A,m,u,\Delta)$ is a non-counital Frobenius algebra.
\end{proof}

Here, $\Delta(1_A)$ is  a {\it Casimir element} of $A$.
This brings us to our first main result.

\begin{theorem} \label{thm:main}
Consider the NSY algebra $A:=B_{n,\ell}(m_0, \dots, m_{n-1})$, with $\kk$-basis $\{X_{i,j}^{r_i, s_{i+j}}\}$ given in Proposition~\ref{prop:Bmi-basis} and algebra structure given in Proposition~\ref{prop:Bmi-alg}. Then, the following statements hold. 
\begin{enumerate} [font=\normalfont]
\item $A$ is non-counital Frobenius with
\begin{align*}
\Delta(X_{i,j}^{r_i, s_{i+j}}) 
&= \displaystyle \sum_{k=0}^{\ell - 1 - j} \; \sum_{t_{i+j+k}=0}^{m_{i+j+k}-1} \; \sum_{t_{i+j+k - \ell +1}=0}^{m_{i+j+k-\ell +1}-1}\Big( 1 - \delta_{m_{i+j+k},m_{i+j+k-\ell+1}}(1 - \delta_{t_{i+j+k}, t_{i+j+k-\ell+1}})\Big)\\ \noalign{\vskip-8pt}
 &  \displaystyle \hspace{2in} \cdot \; X_{i,j+k}^{r_i, t_{i+j+k}}\; \otimes\; X_{i+j+k-\ell+1, \ell - 1-k}^{t_{i+j+k-\ell+1}, s_{i+j}}. 
\end{align*}
\item $(A, \Delta)$ is Frobenius if and only if $m_i = m_{i - \ell +1}$ for all $i=0, \dots n-1$; in which case,  $$\varepsilon(X_{i,j}^{r_i, s_{i+j}}) = \delta_{j, \ell-1} \; \delta_{r_i, s_{i+j}} \; 1_\kk$$ is the counit of $\Delta$.
 \end{enumerate}
\end{theorem}

\begin{proof}
(a) First, note that by Proposition~\ref{prop:Bmi-alg}, 
 \[ 
{\small
\begin{array}{rl}
\Delta(1_A) &=  \displaystyle \sum_ {i=0}^{n-1} \; \sum_{r_i=0}^{m_{i}-1} \Delta(X_{i,0}^{r_i,r_i}) \medskip \\
&= \displaystyle  \sum_ {i=0}^{n-1} \; \sum_{r_i=0}^{m_{i}-1} \; \sum_{k=0}^{\ell - 1} \; \sum_{t_{i+k}=0}^{m_{i+k}-1} \; \sum_{t_{i+k - \ell +1}=0}^{m_{i+k-\ell +1}-1}\Big( 1 - \delta_{m_{i+k},m_{i+k-\ell+1}}(1 - \delta_{t_{i+k}, t_{i+k-\ell+1}})\Big)\\
 \noalign{\vskip-5pt} & \displaystyle \hspace{2.2in} \cdot \; X_{i,k}^{r_i, t_{i+k}}\; \otimes\; X_{i+k-\ell+1, \ell - 1-k}^{t_{i+k-\ell+1}, r_{i}}.
\end{array}
}
\]
It suffices to show  that  \eqref{eq:Delta1} holds for the basis  of $A$  to conclude that  $A$ is non-counital Frobenius [Lemma~\ref{lem:coassoc}]. To  proceed, take the basis element $X_{a,b}^{r_a, s_{a+b}}$ and consider the following computations:
{\small
\begin{align*}
  & (X_{a,b}^{r_a, s_{a+b}} \otimes 1_A)\Delta(1_A) \bigskip \\ & = \displaystyle  \sum_ {i=0}^{n-1} \; \sum_{r_i=0}^{m_{i}-1} \; \sum_{k=0}^{\ell - 1} \; \sum_{t_{i+k}=0}^{m_{i+k}-1} \; \sum_{t_{i+k - \ell +1}=0}^{m_{i+k-\ell +1}-1}\Big( 1 - \delta_{m_{i+k},m_{i+k-\ell+1}}(1 - \delta_{t_{i+k}, t_{i+k-\ell+1}})\Big)\\
\noalign{\vskip-8pt} &  \displaystyle \hspace{2.2in} \cdot \;X_{a,b}^{r_a, s_{a+b}} \; X_{i,k}^{r_i, t_{i+k}}\; \otimes\; X_{i+k-\ell+1, \ell - 1-k}^{t_{i+k-\ell+1}, r_{i}} \bigskip \\
 &  = \displaystyle  \sum_ {i=0}^{n-1} \; \sum_{r_i=0}^{m_{i}-1} \; \sum_{k=0}^{\ell - 1-b} \; \sum_{t_{i+k}=0}^{m_{i+k}-1} \; \sum_{t_{i+k - \ell +1}=0}^{m_{i+k-\ell +1}-1}\Big( 1 - \delta_{m_{i+k},m_{i+k-\ell+1}}(1 - \delta_{t_{i+k}, t_{i+k-\ell+1}})\Big)\\
\noalign{\vskip-8pt}  & \displaystyle \hspace{2.2in} \cdot \;\delta_{i,a+b} \; \delta_{r_i, s_{a+b}}\; X_{a,b+k}^{r_a, t_{i+k}} \; \otimes\; X_{i+k-\ell+1, \ell - 1-k}^{t_{i+k-\ell+1}, r_{i}} \bigskip \\
 &  = \displaystyle   \sum_{k=0}^{\ell - 1-b} \; \sum_{t_{a+b+k}=0}^{m_{a+b+k}-1} \; \sum_{t_{a+b+k - \ell +1}=0}^{m_{a+b+k-\ell +1}-1}\Big( 1 - \delta_{m_{a+b+k},m_{a+b+k-\ell+1}}(1 - \delta_{t_{a+b+k}, t_{a+b+k-\ell+1}})\Big)\\
\noalign{\vskip-8pt} &  \displaystyle \hspace{2in} \cdot  \;  X_{a,b+k}^{r_a, t_{a+b+k}} \; \otimes\; X_{a+b+k-\ell+1, \ell - 1-k}^{t_{a+b+k-\ell+1}, s_{a+b}}.
\end{align*}
}
and 
{\small
\begin{align*}
  &\Delta(1_A) (1_A \otimes X_{a,b}^{r_a, s_{a+b}}) \bigskip \\ & = \displaystyle  \sum_ {i=0}^{n-1} \; \sum_{r_i=0}^{m_{i}-1} \; \sum_{k=0}^{\ell - 1} \; \sum_{t_{i+k}=0}^{m_{i+k}-1} \; \sum_{t_{i+k - \ell +1}=0}^{m_{i+k-\ell +1}-1}\Big( 1 - \delta_{m_{i+k},m_{i+k-\ell+1}}(1 - \delta_{t_{i+k}, t_{i+k-\ell+1}})\Big)\\
 \noalign{\vskip-8pt} & \displaystyle \hspace{2.2in} \cdot \; X_{i,k}^{r_i, t_{i+k}}\; \otimes\; X_{i+k-\ell+1, \ell - 1-k}^{t_{i+k-\ell+1}, r_{i}} \; X_{a,b}^{r_a, s_{a+b}} \bigskip\\
  & = \displaystyle  \sum_ {i=0}^{n-1} \; \sum_{r_i=0}^{m_{i}-1} \; \sum_{k=b}^{\ell - 1} \; \sum_{t_{i+k}=0}^{m_{i+k}-1} \; \sum_{t_{i+k - \ell +1}=0}^{m_{i+k-\ell +1}-1}\Big( 1 - \delta_{m_{i+k},m_{i+k-\ell+1}}(1 - \delta_{t_{i+k}, t_{i+k-\ell+1}})\Big)\\
 \noalign{\vskip-8pt} & \displaystyle \hspace{2.2in} \cdot \; X_{i,k}^{r_i, t_{i+k}}\; \otimes\; \delta_{a,i}\; \delta_{r_a,r_i}\; X_{i+k-\ell+1, \ell - 1-k+b}^{t_{i+k-\ell+1}, s_{a+b}}\bigskip\\
 &  = \displaystyle  \sum_{k=b}^{\ell - 1} \; \sum_{t_{a+k}=0}^{m_{a+k}-1} \; \sum_{t_{a+k - \ell +1}=0}^{m_{a+k-\ell +1}-1}\Big( 1 - \delta_{m_{a+k},m_{a+k-\ell+1}}(1 - \delta_{t_{a+k}, t_{a+k-\ell+1}})\Big)\\
\noalign{\vskip-8pt}  & \displaystyle \hspace{1.8in} \cdot \; X_{a,k}^{r_a, t_{a+k}}\; \otimes\; X_{a+k-\ell+1, \ell - 1-k+b}^{t_{a+k-\ell+1}, s_{a+b}}\bigskip\\
 &  = \displaystyle  \sum_{k=0}^{\ell - 1-b} \; \sum_{t_{a+b+k}=0}^{m_{a+b+k}-1} \; \sum_{t_{a+b+k - \ell +1}=0}^{m_{a+b+k-\ell +1}-1}\Big( 1 - \delta_{m_{a+b+k},m_{a+b+k-\ell+1}}(1 - \delta_{t_{a+b+k}, t_{a+b+k-\ell+1}})\Big)\\
\noalign{\vskip-8pt} &  \displaystyle \hspace{1.8in} \cdot \; X_{a,b+k}^{r_a, t_{a+b+k}}\; \otimes\; X_{a+b+k-\ell+1, \ell - 1-k}^{t_{a+b+k-\ell+1}, s_{a+b}}.
\end{align*}
}
So, $  (X_{a,b}^{r_a, s_{a+b}} \otimes 1_A) \Delta(1_A) =  \Delta(1_A) (1_A \otimes X_{a,b}^{r_a, s_{a+b}})$ as desired.

\bigskip

(b) The first statement follows from the following equivalences:
{\small
\[
\text{$A$ is Frobenius } \overset{\textnormal{Thm.  \ref{thm:SY}}}{\Longleftrightarrow} m_i=m_{\nu(i)}\; \forall\; i\;  \overset{\textnormal{Prop.  \ref{prop:Nak-perm}}}{\iff} m_i=m_{i+\ell-1} \; \forall \; i\; \overset{i\rightarrow i-\ell+1}{\iff} m_i=m_{i-\ell+1} \; \forall \; i.  \]
}

\noindent When $A$ is Frobenius, that is, when $m_i = m_{i-\ell+1}$ for $i=0, \dots, n-1$, we get that 
{\small
\begin{align*}
\Delta(X_{i,j}^{r_i, s_{i+j}}) 
&= \displaystyle \sum_{k=0}^{\ell - 1 - j} \; \sum_{t_{i+j+k}=0}^{m_{i+j+k}-1} \; \sum_{t_{i+j+k - \ell +1}=0}^{m_{i+j+k-\ell +1}-1}\Big( 1 - \delta_{m_{i+j+k},m_{i+j+k-\ell+1}}(1 - \delta_{t_{i+j+k}, t_{i+j+k-\ell+1}})\Big)\\
\noalign{\vskip-8pt} & \displaystyle \hspace{2in} \cdot \; X_{i,j+k}^{r_i, t_{i+j+k}}\; \otimes\; X_{i+j+k-\ell+1, \ell - 1-k}^{t_{i+j+k-\ell+1}, s_{i+j}} \bigskip \\
 &= \displaystyle \sum_{k=0}^{\ell - 1 - j} \; \sum_{t_{i+j+k}=0}^{m_{i+j+k}-1} \; \sum_{t_{i+j+k - \ell +1}=0}^{m_{i+j+k}-1}\delta_{t_{i+j+k}, t_{i+j+k-\ell+1}}  \; X_{i,j+k}^{r_i, t_{i+j+k}}\; \otimes\; X_{i+j+k-\ell+1, \ell - 1-k}^{t_{i+j+k-\ell+1}, s_{i+j}}\bigskip
 \\
 &= \displaystyle \sum_{k=0}^{\ell - 1 - j} \; \sum_{t_{i+j+k}=0}^{m_{i+j+k}-1} \; X_{i,j+k}^{r_i, t_{i+j+k}}\; \otimes\; X_{i+j+k-\ell+1, \ell - 1-k}^{t_{i+j+k}, s_{i+j}}.
\end{align*}
}
Now in the Frobenius case, consider the following computations:
{\small
\begin{align*}
(\varepsilon \otimes \id) \Delta(X_{i,j}^{r_i, s_{i+j}}) 
 &= \displaystyle \sum_{k=0}^{\ell - 1 - j} \; \sum_{t_{i+j+k}=0}^{m_{i+j+k}-1} \; \varepsilon(X_{i,j+k}^{r_i, t_{i+j+k}})\; \otimes\; X_{i+j+k-\ell+1, \ell - 1-k}^{t_{i+j+k}, s_{i+j}} \bigskip\\ 
 &= \displaystyle \sum_{k=0}^{\ell - 1 - j} \; \sum_{t_{i+j+k}=0}^{m_{i+j+k}-1} \delta_{j+k, \ell-1} \; \delta_{r_i, t_{i+j+k}}\; X_{i+j+k-\ell+1, \ell - 1-k}^{t_{i+j+k}, s_{i+j}} \bigskip\\
  &=  X_{i+j+(\ell-1-j)-\ell+1, \ell - 1-(\ell-1-j)}^{r_i, s_{i+j}} \\ 
\noalign{\vskip5pt}  &=  X_{i, j}^{r_i, s_{i+j}}, \\
(\id \otimes \varepsilon) \Delta(X_{i,j}^{r_i, s_{i+j}}) 
 &= \displaystyle \sum_{k=0}^{\ell - 1 - j} \; \sum_{t_{i+j+k}=0}^{m_{i+j+k}-1} \; X_{i,j+k}^{r_i, t_{i+j+k}}\; \otimes\; \varepsilon(X_{i+j+k-\ell+1, \ell - 1-k}^{t_{i+j+k}, s_{i+j}}) \bigskip\\ 
 &= \displaystyle \sum_{k=0}^{\ell - 1 - j} \; \sum_{t_{i+j+k}=0}^{m_{i+j+k}-1} \delta_{\ell-1-k, \ell-1} \; \delta_{t_{i+j+k},s_{i+j}}\; X_{i,j+k}^{r_i, t_{i+j+k}}\bigskip\\
  &=  X_{i,j}^{r_i, s_{i+j}}.
\end{align*}
}

\noindent Therefore, $\Delta$ is counital  via $\varepsilon$ in the Frobenius case, as claimed.


\vspace{-.1in}
\end{proof}


 \section{Examples for NSY algebras} \label{sec:examples}
In this section, we illustrate the results of the previous section for various examples of the (self-injective) NSY algebras, $B_{n,\ell}(m_0, \dots,m_{n-1})$ [Definition~\ref{def:NSY}]. First, by  Proposition~\ref{prop:Nak-perm} and Theorem~\ref{thm:SY}, we have the following statement: 
\begin{equation} \label{eq:FrobSY}
\text{$B_{n,\ell}(m_0, \dots,m_{n-1})$ is Frobenius \quad $\iff$ \quad $m_i = m_{i+\ell-1}$ for all $i$.}
\end{equation}
In Section~\ref{sec:FrobEx}, we analyze the Frobenius NSY algebras $B_{n,\ell}(1, \dots, 1)$. In Section~\ref{sec:Nakayama}, we study Nakayama's example of a non-Frobenius, self-injective algebra, the NSY algebra $B_{2,2}(2,1)$ \cite[page~624]{Nak}.  We end by comparing the Frobenius algebra $B_{4,3}(1,2,1,2)$ with the non-Frobenius, self-injective algebra $B_{4,3}(1,1,2,2)$ in Section~\ref{sec:highDim}.

 \pagebreak
 
\subsection{Frobenius examples} \label{sec:FrobEx}

\subsubsection{The NSY algebra $B_{2,2}(1,1)$} We begin with a discussion of the example $B_{2,2}(1,1)$. By Lemma~\ref{lem:Bmi-dim} and Proposition~\ref{prop:Bmi-basis},
$B_{2,2}(1,1)$ is a 4-dimensional algebra with $\kk$-basis:
$$\{ X_{0,0}^{0,0}, \; \; X_{0,1}^{0, 0}, \; \; X_{1,0}^{0,0},\; \;  X_{1,1}^{0, 0} \}.$$

\noindent By Proposition~\ref{prop:Bmi-alg}, the multiplication of this algebra is given by the following table:

\medskip

{\small
\begin{center}
\renewcommand\arraystretch{1.3}
\setlength\doublerulesep{0pt}
\begin{tabular}{r||*{4}{2|}}
* & X_{0,0}^{0,0} & X_{0,1}^{0, 0} & X_{1,0}^{0,0} & X_{1,1}^{0, 0} \\
\hline\hline
$X_{0,0}^{0,0}$ & X_{0,0}^{0,0} & X_{0,1}^{0, 0} & 0 & 0 \\ 
\hline
$X_{0,1}^{0, 0}$ & 0 & 0 & X_{0,1}^{0, 0} & 0 \\ 
\hline
$X_{1,0}^{0,0}$ & 0 & 0 & X_{1,0}^{0,0} & X_{1,1}^{0,0} \\ 
\hline
$X_{1,1}^{0, 0}$ & X_{1,1}^{0,0} & 0 & 0 & 0 \\ 
\hline
\end{tabular}
\end{center}
}

\bigskip

\noindent Moreover, the unit of $B_{2,2}(1,1)$ is $X_{0,0}^{0,0} + X_{1,0}^{0,0}$.

Now, by Theorem~\ref{thm:main}(a), we compute the comultiplication  formula as follows: 
{\small
\begin{align*}
    \Delta(X_{0,0}^{0,0}) &= \sum_{k=0}^{1 } \; \sum_{t_k=0}^{m_{k}-1} \; \sum_{t_{k - 1}=0}^{m_{k-1}-1}\Big( 1 - \delta_{m_{k},m_{k-1}}(1 - \delta_{t_{k}, t_{k-1}})\Big) \cdot \; X_{0,k}^{0, t_{k}}\; \otimes\; X_{k-1, 1-k}^{t_{k-1}, 0} \\
 &= \sum_{k=0}^{1 } \; \sum_{t_k=0}^{0} \; \sum_{t_{k - 1}=0}^{0} \Big( 1 - \delta_{1,1}(1 - \delta_{t_{k}, t_{k-1}})\Big) \cdot \; X_{0,k}^{0, t_{k}}\; \otimes\; X_{k-1, 1-k}^{t_{k-1}, 0} \\
 &= \sum_{k=0}^{1 } X_{0,k}^{0, 0} \; \otimes\; X_{k-1, 1-k}^{0, 0} \\
 &= X_{0,0}^{0,0} \otimes X_{1,1}^{0,0} + X_{0,1}^{0,0} \otimes X_{0,0}^{0,0} \;\;;\\
 \\
\Delta(X_{0,1}^{0,0}) &= \sum_{k=0}^{0} \; \sum_{t_{k+1}=0}^{m_{1+k}-1} \; \sum_{t_k=0}^{m_{k}-1}\Big( 1 - \delta_{m_{1+k},m_{k}}(1 - \delta_{t_{1+k}, t_{k}})\Big) \cdot \; X_{0,1+k}^{0, t_{1+k}}\; \otimes\; X_{k, 1-k}^{t_{k}, 0} \\
 &=\sum_{k=0}^{0 } \; \sum_{t_k=0}^{0} \; \sum_{t_{k - 1}=0}^{0} \Big( 1 - \delta_{1,1} \left(1 - \delta_{t_{1+k}, t_{k}} \right) \Big) \cdot \; X_{0,k+1}^{0, t_{k+1}}\; \otimes\; X_{k, 1-k}^{t_{k}, 0} \\
 &= X_{0,1}^{0,0} \otimes X_{0,1}^{0,0} \;\;; \\
 \\
\Delta(X_{1,0}^{0, 0}) 
&=  \sum_{k=0}^{1} \; \sum_{t_{1+k}=0}^{m_{1+k}-1} \; \sum_{t_k=0}^{m_{k}-1}\Big( 1 - \delta_{m_{k+1},m_{k}}(1 - \delta_{t_{1+k}, t_{k}})\Big)  \cdot \; X_{1,k}^{0, t_{1+k}}\; \otimes\; X_{k, 1-k}^{t_{k}, 0} \\
&=  \sum_{k=0}^{1} \Big( 1 - (1 - \delta_{0,0})\Big)  \cdot \; X_{1,k}^{0, 0}\; \otimes\; X_{k, 1-k}^{0, 0} \\
&=  X_{1,0}^{0, 0} \otimes X_{0, 1}^{0, 0} + X_{1,1}^{0, 0} \otimes X_{1, 0}^{0, 0} \;\;;\\
\\
\Delta( X_{1,1}^{0,0}) &= \sum_{k=0}^{0} \; \sum_{t_{2+k}=0}^{m_{2+k}-1} \; \sum_{t_{k+1}=0}^{m_{k+1}-1}\Big( 1 - \delta_{m_{2+k},m_{k+1}}(1 - \delta_{t_{2+k}, t_{k+1}})\Big)\cdot \; X_{1,1+k}^{0, t_{2+k}}\; \otimes\; X_{k+1, 1-k}^{t_{k+1}, 0}\\
&= \sum_{k=0}^{0} \; \sum_{t_{2+k}=0}^{0} \; \sum_{t_{k+1}=0}^{0}\Big( 1 - \delta_{1,1}(1 - \delta_{t_{2+k}, t_{k+1}})\Big)\cdot \; X_{1,1+k}^{0, t_{2+k}}\; \otimes\; X_{k+1, 1-k}^{t_{k+1}, 0} \\
&= \sum_{k=0}^{0} \Big( 1 - (1 - \delta_{0, 0})\Big)\cdot \; X_{1,1+k}^{0, 0}\; \otimes\; X_{k+1, 1-k}^{0, 0} \\
&= X_{1,1}^{0, 0}\; \otimes\; X_{1, 1}^{0, 0}.
\end{align*}
}

\noindent This gives $B_{2,2}(1,1)$ the structure of a non-counital Frobenius algebra. By~\eqref{eq:FrobSY} and Theorem~\ref{thm:main}(b), the counit for this algebra exists and is given by 
$$ \varepsilon(X_{0,0}^{0, 0}) \;= \; \varepsilon(X_{1,0}^{0, 0})\; =\; 0_\kk,  \quad\quad \varepsilon(X_{0,1}^{0, 0})\; = \; \varepsilon(X_{1,1}^{0, 0})\; =\; 1_\kk.$$

\subsubsection{The NSY algebras $B_{n,\ell}(1,\dots,1)$} We can  generalize the work in the previous section by considering the NSY algebras $B_{n,\ell}(1,\dots,1)$. By Lemma~\ref{lem:Bmi-dim} and Proposition~\ref{prop:Bmi-basis}, $B_{n,\ell}(1,\dots,1)$ is a $n \ell$-dimensional algebra with basis given by $\{X_{i,j}^{0,0}\}_{i,j}$. The general comultiplication formula for this case simplifies to:
{\small
\begin{align*}
\Delta(X_{i,j}^{0, 0}) 
&= \displaystyle \sum_{k=0}^{\ell - 1 - j} \; \sum_{t_{i+j+k}=0}^{m_{i+j+k}-1} \; \sum_{t_{i+j+k - \ell +1}=0}^{m_{i+j+k-\ell +1}-1}\Big( 1 - \delta_{m_{i+j+k},m_{i+j+k-\ell+1}}(1 - \delta_{t_{i+j+k}, t_{i+j+k-\ell+1}})\Big)\\
\noalign{\vskip-8pt} & \displaystyle \hspace{2in} \cdot \; X_{i,j+k}^{0, t_{i+j+k}}\; \otimes\; X_{i+j+k-\ell+1, \ell - 1-k}^{t_{i+j+k-\ell+1}, 0} \\
 &= \displaystyle \sum_{k=0}^{\ell - 1 - j} \; \sum_{t_{i+j+k}=0}^{0} \; \sum_{t_{i+j+k - \ell +1}=0}^{0}\Big( 1 - \delta_{0,0}(1 - \delta_{t_{i+j+k}, t_{i+j+k-\ell+1}})\Big)\\
\noalign{\vskip-8pt} & \displaystyle \hspace{2in} \cdot \; X_{i,j+k}^{0, t_{i+j+k}}\; \otimes\; X_{i+j+k-\ell+1, \ell - 1-k}^{t_{i+j+k-\ell+1}, 0} \\
 &= \displaystyle \sum_{k=0}^{\ell - 1 - j} \; X_{i,j+k}^{0, 0}\; \otimes\; X_{i+j+k-\ell+1, \ell - 1-k}^{0, 0}. 
\end{align*}
}
\noindent Moreover, the counit formula simplifies to $\varepsilon(X_{i,j}^{0, 0}) \;=  \delta_{j, \ell-1} \; 1_\kk.$

\smallskip

The reader may wish to compare these results with work of Wang-Zhang in \cite{WangZhang}.

\smallskip

\subsection{Nakayama's non-Frobenius self-injective algebra}
\label{sec:Nakayama} 
Next, we consider Nakayama's example of a finite-dimensional self-injective algebra, that is not Frobenius \cite[page~624]{Nak}. This algebra is 9-dimensional (see~Lemma~\ref{lem:Bmi-dim}) and is isomorphic to the NSY algebra $B_{2,2}(2,1)$; see  \cite[Section~3]{SYpaper} for a proof. By Proposition~\ref{prop:Bmi-basis}, the $\kk$-basis of 
$B_{2,2}(2,1)$ is
$$\{ X_{0,0}^{0,0}, \; \; X_{0,0}^{0,1},  \; \; X_{0,0}^{1,0}, \; \; X_{0,0}^{1,1}, \; \; X_{0,1}^{0,0},  \; \; X_{0,1}^{1,0}, \; \; X_{1,0}^{0,0}, \; \; X_{1,1}^{0,0}, \; \; X_{1,1}^{0,1}\},$$ 
Via the multiplication table below, these basis elements correspond, respectively, to the basis elements of Nakayama's algebra in \cite{Nak}:
$$\{ \alpha_{11}, \; \; \alpha_{12},  \; \; \alpha_{21}, \; \; \alpha_{22}, \; \; \gamma_1,  \; \; \gamma_2, \; \; \beta, \; \; \delta_1, \; \; \delta_2\}.$$
\noindent Now by Proposition~\ref{prop:Bmi-alg}, the multiplication of $B_{2,2}(2,1)$ is given by the following table:

\medskip

{\small
\begin{center}
\renewcommand\arraystretch{1.3}
\setlength\doublerulesep{0pt}
\begin{tabular}{r||*{9}{2|}}
* & X_{0,0}^{0,0} & X_{0,0}^{0, 1} & X_{0,0}^{1,0} & X_{0,0}^{1, 1} & X_{0,1}^{0, 0} & X_{0,1}^{1, 0} & X_{1,0}^{0, 0} & X_{1,1}^{0, 0} & X_{1,1}^{0, 1} \\
\hline\hline
$X_{0,0}^{0,0}$ & X_{0,0}^{0,0} & X_{0,0}^{0,1} & 0 & 0 & X_{0,1}^{0,0} & 0 & 0 & 0 & 0 \\ 
\hline
$X_{0,0}^{0,1}$ & 0 & 0 & X_{0,0}^{0,0} & X_{0,0}^{0,1} & 0 & X_{0,1}^{0,0} & 0 & 0 & 0 \\ 
\hline
$X_{0,0}^{1,0}$ & X_{0,0}^{1,0} & X_{0,0}^{1,1} & 0 & 0 & X_{0,1}^{1,0} & 0 & 0 & 0 & 0 \\ 
\hline
$X_{0,0}^{1,1}$ & 0 & 0 & X_{0,0}^{1,0} & X_{0,0}^{1,1} & 0 & X_{0,1}^{1,0} & 0 & 0 & 0 \\ 
\hline
$X_{0,1}^{0,0}$ & 0 & 0 & 0 & 0 & 0 & 0 & X_{0,1}^{0,0} & 0 & 0 \\ 
\hline
$X_{0,1}^{1,0}$ & 0 & 0 & 0 & 0 & 0 & 0 & X_{0,1}^{1,0} & 0 & 0 \\ 
\hline
$X_{1,0}^{0,0}$ & 0 & 0 & 0 & 0 & 0 & 0 & X_{1,0}^{0,0} & X_{1,1}^{0,0} & X_{1,1}^{0,1} \\ 
\hline
$X_{1,1}^{0,0}$ & X_{1,1}^{0,0} & X_{1,1}^{0,1} & 0 & 0 & 0 & 0 & 0 & 0 & 0 \\ 
\hline
$X_{1,1}^{0,1}$ & 0 & 0 & X_{1,1}^{0,0} & X_{1,1}^{0,1} & 0 & 0 & 0 & 0 & 0 \\ 
\hline
\end{tabular}
\end{center}
}

\medskip

\noindent Moreover, the unit of $B_{2,2}(2,1)$ is $X_{0,0}^{0,0} + X_{0,0}^{1,1} + X_{1,0}^{0,0}$.

By Theorem~\ref{thm:main}(a), we compute the comultiplication for $B_{2,2}(2,1)$ as follows:
{\small $$\Delta(X_{i,j}^{r_i, s_{i+j}}) 
= \displaystyle \sum_{k=0}^{1 - j} \; \sum_{t_{i+j+k}=0}^{m_{i+j+k}-1} \; \sum_{t_{i+j+k - 1}=0}^{m_{i+j+k-1}-1} \; X_{i,j+k}^{r_i, t_{i+j+k}}\; \otimes\; X_{i+j+k-1, 1-k}^{t_{i+j+k-1}, s_{i+j}}.$$}
So we get the output below:
{\small
\begin{align*}
\Delta(X_{0,0}^{0,0}) &= \textstyle \sum_{k=0}^{1}  \sum_{t_{k}=0}^{m_{k}-1} \sum_{t_{k - 1}=0}^{m_{k-1}-1} X_{0,k}^{0, t_{k}} \otimes X_{k-1, 1-k}^{t_{k-1}, 0} \\
&= X_{0,0}^{0, 0} \otimes X_{1, 1}^{0, 0} 
+ X_{0,0}^{0, 1} \otimes X_{1, 1}^{0, 0} 
+ X_{0,1}^{0, 0} \otimes X_{0, 0}^{0, 0} 
+ X_{0,1}^{0, 0} \otimes X_{0, 0}^{1, 0} ;\\
\noalign{\vskip5pt} \Delta(X_{0,0}^{0,1}) &= X_{0,0}^{0, 0} \otimes X_{1, 1}^{0, 1} 
+ X_{0,0}^{0, 1} \otimes X_{1, 1}^{0, 1} 
+ X_{0,1}^{0, 0} \otimes X_{0, 0}^{0, 1} 
+ X_{0,1}^{0, 0} \otimes X_{0, 0}^{1, 1} ;\\
\noalign{\vskip5pt}
\Delta(X_{0,0}^{1,0}) &= X_{0,0}^{1, 0} \otimes X_{1, 1}^{0, 0} 
+ X_{0,0}^{1, 1} \otimes X_{1, 1}^{0, 0} 
+ X_{0,1}^{1, 0} \otimes X_{0, 0}^{0, 0} 
+ X_{0,1}^{1, 0} \otimes X_{0, 0}^{1, 0} ;\\
\noalign{\vskip5pt}
\Delta(X_{0,0}^{1,1}) &= X_{0,0}^{1, 0} \otimes X_{1, 1}^{0, 1} 
+ X_{0,0}^{1, 1} \otimes X_{1, 1}^{0, 1} 
+ X_{0,1}^{1, 0} \otimes X_{0, 0}^{0, 1} 
+ X_{0,1}^{1, 0} \otimes X_{0, 0}^{1, 1} ;\\
\noalign{\vskip5pt}
\Delta(X_{0,1}^{0,0}) &= \textstyle \sum_{k=0}^{0}  \sum_{t_{k+1}=0}^{m_{k+1}-1} \sum_{t_{k}=0}^{m_{k}-1} X_{0,k+1}^{0, t_{k+1}} \otimes X_{k, 1-k}^{t_{k}, 0}\\
&= X_{0,1}^{0,0} \otimes X_{0, 1}^{0, 0} + X_{0,1}^{0,0} \otimes X_{0, 1}^{1, 0} ;\\
\noalign{\vskip5pt}
\Delta(X_{0,1}^{1,0}) &= X_{0,1}^{1,0} \otimes X_{0, 1}^{0, 0} + X_{0,1}^{1,0} \otimes X_{0, 1}^{1, 0} ;\\
\noalign{\vskip5pt}
\Delta(X_{1,0}^{0,0}) &= \textstyle \sum_{k=0}^{1}  \sum_{t_{k+1}=0}^{m_{k+1}-1} \sum_{t_{k}=0}^{m_{k}-1} X_{1,k}^{0, t_{k+1}} \otimes X_{k, 1-k}^{t_{k}, 0}\\
&=X_{1,0}^{0,0} \otimes X_{0,1}^{0, 0}
+ X_{1,0}^{0,0} \otimes X_{0, 1}^{1, 0}
+ X_{1,1}^{0,0} \otimes X_{1, 0}^{0, 0}
+ X_{1,1}^{0,1} \otimes X_{1, 0}^{0, 0} ;\\
\noalign{\vskip5pt}
\Delta(X_{1,1}^{0,0}) &= \textstyle \sum_{k=0}^{0}  \sum_{t_{k}=0}^{m_{k}-1} \sum_{t_{k+1}=0}^{m_{k+1}-1} X_{1,k+1}^{0, t_{k}} \otimes X_{k+1, 1-k}^{t_{k+1}, 0}\\
&=X_{1,1}^{0, 0} \otimes X_{1, 1}^{0, 0}+X_{1,1}^{0, 1} \otimes X_{1,1}^{0, 0} ;\\
\noalign{\vskip5pt}
\Delta(X_{1,1}^{0,1}) &=X_{1,1}^{0, 0} \otimes X_{1, 1}^{0, 1} + X_{1,1}^{0, 1} \otimes X_{1,1}^{0, 1}.
\end{align*}
}

Lastly, we see that the comultiplication is not counital using $\varepsilon$ in Theorem~\ref{thm:main}(b):
{\small
\begin{align*}
(\varepsilon \otimes \id) \Delta(X_{0,0}^{0, 0}) &= \varepsilon(X_{0,0}^{0, 0}) \; X_{1, 1}^{0, 0} 
+ \varepsilon(X_{0,0}^{0, 1}) \;  X_{1, 1}^{0, 0} 
+ \varepsilon(X_{0,1}^{0, 0}) \;  X_{0, 0}^{0, 0} 
+ \varepsilon(X_{0,1}^{0, 0}) \;  X_{0, 0}^{1, 0} =  X_{0, 0}^{0, 0} +  X_{0, 0}^{1, 0},\\
(\id \otimes \varepsilon) \Delta(X_{0,0}^{0, 0}) &= X_{0,0}^{0, 0} \; \varepsilon(X_{1, 1}^{0, 0}) 
+ X_{0,0}^{0, 1} \;  \varepsilon(X_{1, 1}^{0, 0}) 
+ X_{0,1}^{0, 0} \;  \varepsilon(X_{0, 0}^{0, 0}) 
+ X_{0,1}^{0, 0} \;  \varepsilon(X_{0, 0}^{1, 0}) =  X_{0, 0}^{0, 0} +  X_{0, 0}^{0, 1}.
\end{align*}
}

\subsection{Higher-dimensional examples}\label{sec:highDim}
In this section we discuss two examples of NSY algebras of dimensions 28 and 27, respectively in Sections~\ref{sec:1212} and~\ref{sec:1122}, and provide examples of the comultiplication formula for these cases.

\subsubsection{The NSY algebras $B_{4,3}(1,2,1,2)$}
\label{sec:1212}
By Lemma~\ref{lem:Bmi-dim} and Proposition~\ref{prop:Bmi-basis}, we get that
{\small
\begin{align*}
\textnormal{dim}(B_{4,3}(1,2,1,2))  &=\sum_{i=0}^3\sum_{j=0}^2 m_{i} m_{i+j}\\
&= 1(1+2+1) + 2(2+1+2) + 1(1+2+1) + 2(2+1+2) = 28.
\end{align*}
}

\noindent Since $m_i=m_{i+\ell-1}$ for all $i$, $B_{4,3}(1,2,1,2)$ is a Frobenius algebra by~\eqref{eq:FrobSY}. Thus, using Theorem~\ref{thm:main}(a) the comultiplication formula simplifies to
{\small
\begin{align*}
 \Delta(X_{i,j}^{r_i, s_{i+j}}) &= \sum_{k=0}^{\ell - 1 - j} \; \sum_{t_{i+j+k}=0}^{m_{i+j+k}-1} \; X_{i,j+k}^{r_i, t_{i+j+k}}\; \otimes\; X_{i+j+k-\ell+1, \ell - 1-k}^{t_{i+j+k}, s_{i+j}} \\
 & =  \sum_{k=0}^{2 - j} \; \sum_{t_{i+j+k}=0}^{m_{i+j+k}-1} \; X_{i,j+k}^{r_i, t_{i+j+k}}\; \otimes\; X_{i+j+k-2, 2-k}^{t_{i+j+k}, s_{i+j}}\; .
\end{align*}
}

\noindent
Below we provide the comultiplication formula for two basis elements $X_{0,0}^{0,0}$ and $X_{2,1}^{0,1}$.
{\small
\begin{align*}
    \Delta(X_{0,0}^{0,0}) &= \sum_{k=0}^2 \sum_{t_k=0}^{m_k -1} X_{0,k}^{0,t_k} \otimes X_{k-2,2-k}^{t_k,0} \\
    &= X_{0,0}^{0,0}\otimes X_{2,2}^{0,0} + X_{0,1}^{0,0}\otimes X_{3,1}^{0,0} + X_{0,1}^{0,1}\otimes X_{3,1}^{1,0} + X_{0,2}^{0,0}\otimes X_{0,0}^{0,0} \;\;; \\
    \Delta(X_{2,1}^{0,1}) & = \sum_{k=0}^{1} \sum_{t_{3+k}=0}^{m_{3+k}-1} X_{2,1+k}^{0,t_{3+k}} \otimes X_{k+1,2-k}^{t_{3+k},1} \\
    & = X_{2,1}^{0,0}\otimes X_{1,2}^{0,1} + X_{2,1}^{0,1}\otimes X_{1,2}^{1,1} + X_{2,2}^{0,0}\otimes X_{2,1}^{0,1} \;\;.
\end{align*} 
}

\noindent Next, we show that the comultiplication is counital for these basis elements using Theorem~\ref{thm:main}(b):
{\small
\begin{align*}
    (\id\otimes \varepsilon) \Delta (X_{0,0}^{0,0}) & = X_{0,0}^{0,0} \; \varepsilon(X_{2,2}^{0,0}) + X_{0,1}^{0,0} \; \varepsilon(X_{3,1}^{0,0}) + X_{0,1}^{0,1} \; \varepsilon(X_{3,1}^{1,0}) + X_{0,2}^{0,0} \; \varepsilon(X_{0,0}^{0,0}) = X^{0,0}_{0,0} ; \\
    (\varepsilon \otimes \id)\Delta(X_{0,0}^{0,0}) &= \varepsilon( X_{0,0}^{0,0})  \; X_{2,2}^{0,0} +  \varepsilon(X_{0,1}^{0,0})  \; X_{3,1}^{0,0} +  \varepsilon(X_{0,1}^{0,1})  \; X_{3,1}^{1,0} +   \varepsilon(X_{0,2}^{0,0})  \; X_{0,0}^{0,0} = X^{0,0}_{0,0} ;\\
   \noalign{\vskip8pt} (\id\otimes \varepsilon)\Delta(X_{2,1}^{0,1}) &= X_{2,1}^{0,0}  \; \varepsilon(X_{1,2}^{0,1}) + X_{2,1}^{0,1} \; \varepsilon(X_{1,2}^{1,1}) + X_{2,2}^{0,0} \; \varepsilon(X_{2,1}^{0,1}) = X_{2,1}^{0,1} ;\\
    (\varepsilon\otimes \id)\Delta(X_{2,1}^{0,1}) &=   \varepsilon(X_{2,1}^{0,0})  \; X_{1,2}^{0,1} +   \varepsilon(X_{2,1}^{0,1})  \; X_{1,2}^{1,1} +   \varepsilon(X_{2,2}^{0,0})  \; X_{2,1}^{0,1}  = X_{2,1}^{0,1}.
\end{align*}
}

We will compare this Frobenius algebra with the non-counital Frobenius algebra, \linebreak $B_{4,3}(1,1,2,2)$, in the next section.

\subsubsection{The NSY algebras $B_{4,3}(1,1,2,2)$}
\label{sec:1122}
By Lemma~\ref{lem:Bmi-dim} and Proposition~\ref{prop:Bmi-basis}, we get that
{\small
\begin{align*}
    \textnormal{dim}(B_{4,3}(1,1,2,2)) & = \sum_{i=0}^3\sum_{j=0}^2 m_{i} m_{i+j}\\
    &= 1(1+1+2)+1(1+2+2)+2(2+2+1)+2(2+1+1)= 27.
\end{align*}
}

\noindent Since $1 = m_0 \neq m_{0+\ell-1}=m_{2}=2$, $B_{4,3}(1,1,2,2)$ is not a Frobenius algebra by~\eqref{eq:FrobSY}. Below we provide an example showing that the comultiplication $\Delta$ is not counital (via $\varepsilon$ in Theorem~\ref{thm:main}(b)) in this case. By Theorem~\ref{thm:main}(a), we have that
{\small
\begin{align*} 
    \Delta(X_{2,1}^{0,1}) = \sum_{k=0}^{1} \; \sum_{t_{k+3}=0}^{m_{k+3}-1} \; \sum_{t_{k+1}=0}^{m_{k+1}-1}\Big( 1 - \delta_{m_{k+3},m_{k+1}}(1 - \delta_{t_{k+3}, t_{k+1}})\Big) \; X_{2,1+k}^{0, t_{3+k}}\; \otimes\; X_{k+1, 2-k}^{t_{k+1}, 1} 
\end{align*}
}
Since $m_{k+3}\neq m_{k+1}$ for all $k$, we get that
{\small
\begin{align*}
    \Delta(X_{2,1}^{0,1}) & = \sum_{k=0}^{1} \; \sum_{t_{k+3}=0}^{m_{k+3}-1} \; \sum_{t_{k+1}=0}^{m_{k+1}-1} X_{2,1+k}^{0, t_{3+k}}\; \otimes\; X_{k+1, 2-k}^{t_{k+1}, 1} \\
    & = X_{2,1}^{0,0}\otimes X_{1,2}^{0,1} + X_{2,1}^{0,1}\otimes X_{1,2}^{0,1} + X_{2,2}^{0,0}\otimes X_{2,1}^{0,1} + X_{2,2}^{0,0}\otimes X_{2,1}^{1,1}. 
\end{align*}
}

\noindent Compare this with the comultiplication for $B_{4,3}(1,2,1,2)$ in the previous section. We can now see that $\Delta$ is not counital because
{\small
\begin{align*}
     (\varepsilon \otimes \id)\Delta(X_{2,1}^{0,1}) &= \varepsilon(X_{2,1}^{0,0}) \;  X_{1,2}^{0,1} + \varepsilon(X_{2,1}^{0,1}) \;  X_{1,2}^{0,1} + \varepsilon(X_{2,2}^{0,0})  \; X_{2,1}^{0,1} + \varepsilon(X_{2,2}^{0,0}) \;  X_{2,1}^{1,1} =X_{2,1}^{0,1}+ X_{2,1}^{1,1},  \\ (\id\otimes\varepsilon)\Delta(X_{2,1}^{0,1}) &= X_{2,1}^{0,0} \; \varepsilon(X_{1,2}^{0,1}) + X_{2,1}^{0,1} \; \varepsilon(X_{1,2}^{0,1}) + X_{2,2}^{0,0} \; \varepsilon(X_{2,1}^{0,1}) + X_{2,2}^{0,0} \; \varepsilon(X_{2,1}^{1,1}) 
    = 0.
\end{align*}
}

\section{Finite dimensional weak Hopf algebras are non-counital Frobenius}
\label{sec:weak}

In this section, we establish an explicit non-counital Frobenius structure for finite-dimen-sional (f.d.) weak Hopf algebras [Definition~\ref{def:weak}], thereby proving Conjecture~\ref{main-conj} for another large class of self-injective algebras. Background material is provided in Section~\ref{sec:weak-backgr}, and the main result is presented in Section~\ref{sec:weak-main}.

\subsection{Background on weak Hopf algebras} \label{sec:weak-backgr}
The following material is from \cite{BNS}.

\begin{definition} \label{def:wba} 
A \textit{weak bialgebra} over $\kk$ is a quintuple $(H,m,u,\Delta_{\textnormal{wk}}, \varepsilon_{\textnormal{wk}})$ such that
\begin{enumerate}[label=(\roman*)]
    \item $(H,m,u)$ is a $\kk$-algebra,
    \smallskip
    \item $(H, \;\Delta_{\textnormal{wk}},\; \varepsilon_{\textnormal{wk}})$ is a $\kk$-coalgebra,\smallskip
    \item \label{def:wba3} $\Delta_{\textnormal{wk}}(ab)=\Delta_{\textnormal{wk}}(a)\;\Delta_{\textnormal{wk}}(b)$ for all $a,b \in H$,\smallskip
    \item \label{def:wba4} $\varepsilon_{\textnormal{wk}}(abc)=\varepsilon_{\textnormal{wk}}(ab_1)\;\varepsilon_{\textnormal{wk}}(b_2c)=\varepsilon_{\textnormal{wk}}(ab_2)\;\varepsilon_{\textnormal{wk}}(b_1c)$ for all $a,b,c \in H$, \smallskip
    \item \label{def:wba5} $\Delta_{\textnormal{wk}}^2(1_H)=(\Delta_{\textnormal{wk}}(1_H) \otimes 1_H)(1_H \otimes \Delta_{\textnormal{wk}}(1_H))=(1_H \otimes \Delta_{\textnormal{wk}}(1_H))(\Delta_{\textnormal{wk}}(1_H) \otimes 1_H)$.\smallskip
\end{enumerate}
Here, we use the sumless Sweedler notation, for $h \in H$: $$\Delta_{\textnormal{wk}}(h):= h_1 \otimes h_2.$$
\end{definition}

\begin{definition}[$\varepsilon_s$, $\varepsilon_t$, $H_s$, $H_t$] \label{def:eps}
Let $(H, m, u, \Delta_{\textnormal{wk}}, \varepsilon_{\textnormal{wk}})$ be a weak bialgebra. We define the {\it source and target counital maps}, respectively as follows:
\[
\begin{array}{ll}
    \varepsilon_s: H \to H, & x \mapsto 1_1\;\varepsilon_{\textnormal{wk}}(x1_2) \\
    \varepsilon_t: H \to H,  & x \mapsto \varepsilon_{\textnormal{wk}}(1_1x)\;1_2.
\end{array}
\]
We denote the images of these maps as $H_s:=\varepsilon_s(H)$ and $H_t:=\varepsilon_t(H)$, which are subalgebras of $H$: the \emph{source counital subalgebra} and  the \emph{target counital subalgebra} of $H$, respectively.
\end{definition}

\begin{definition} \label{def:weak}
A \textit{weak Hopf algebra} is a sextuple $(H,m,u,\Delta_\textnormal{wk},\varepsilon_\textnormal{wk}, S)$, where the quintuple $(H,m,u,\Delta_\textnormal{wk},\varepsilon_\textnormal{wk})$ is a weak bialgebra and $S: H \to H$ is a $\kk$-linear map called the \textit{antipode} that satisfies the following properties for all $h\in H$:
 $$S(h_1)h_2=\varepsilon_s(h), \quad \quad
h_1S(h_2)=\varepsilon_t(h), \quad \quad
S(h_1)h_2S(h_3)=S(h).$$
\end{definition}

It follows that $S$ is anti-multiplicative with respect to $m$,  and anti-comultiplicative with respect to $\Delta_\textnormal{wk}$.

\begin{definitionproposition}\cite[page~5]{BNS} \label{def:Hopf}
Take a weak Hopf algebra $H$. Then the following conditions are equivalent:
\begin{enumerate}
    \item  $\Delta_\textnormal{wk}(1_H)=1_H\otimes 1_H$;\smallskip
 \item $\ep_\textnormal{wk}(xy)=\ep_\textnormal{wk}(x)\;\ep_\textnormal{wk}(y)$ for all $x,y\in H$;\smallskip
\item $S(x_1)x_2=\ep_\textnormal{wk}(x)\;1_H$ for all $x \in H$; and \smallskip
 \item $x_1S(x_2)=\ep_\textnormal{wk}(x)\;1_H$ for all $x \in H$.\smallskip
 \end{enumerate}
In this case,  $H$ is  a {\it Hopf algebra}.  \qed
 \end{definitionproposition}

\begin{hypothesis}
Recall we assume that all algebras in this work are finite-dimensional, and we will continue to assume this for weak Hopf algebras. 
\end{hypothesis}

\begin{remark}
 Here, $H^*$ will be the usual $\kk$-linear dual of $H$, which admits the structure of  a weak Hopf algebra \cite[page~5]{BNS}.
\end{remark}

Now we consider an important set of elements whose existence will determine when a finite-dimensional weak Hopf algebra is Frobenius.

\begin{definition} \label{def:integral}
Let $H$ be a weak Hopf algebra. 
\begin{enumerate}
    \item An element $\Lambda$ in $H$ is called a {\it left} (resp., {\it right}) {\it integral} if $h \Lambda = \ep_t(h) \Lambda$ (resp., $\Lambda h= \Lambda\ep_s(h)$) for all $h \in H$. 
    \item Let $I^L(H)$ (resp., $I^R(H)$) denote the space of left (resp., right) integrals of $H$. \smallskip
    \item  A left/right integral $\Lambda \in H$ is called {\it non-degenerate} if the linear map
    \begin{equation} \label{eq:psi-bij}
     \Psi_{\Lambda}: H^* \to H, \; \phi \mapsto  \Lambda_1 \phi(\Lambda_2)
    \end{equation}
    is a bijection.  
\end{enumerate}
\end{definition}

\begin{remark} \label{rem:nondeg}
\begin{enumerate}
\item Note that the map $\Psi_\Lambda$ above is bijective if and only if the map 
$$\Phi_\Lambda : H^* \xrightarrow{(\Psi_\Lambda)^*} H^{**} \xrightarrow{\sim} H \xrightarrow{S} H, \; \; \phi \mapsto \phi(\Lambda_1)S(\Lambda_2)$$ 
is bijective, as the antipode of a finite-dimensional weak Hopf algebra is bijective \cite[Theorem~2.10]{BNS}. Moreover, this occurs if and only if the composition below is bijective:
$$\Phi'_\Lambda: H^*\xrightarrow{(\Phi_\Lambda)^*} H^{**} \xrightarrow{\sim} H, \; \; \phi \mapsto \Lambda_1 \phi(S(\Lambda_2)).$$
\item Note that $(\Psi_\Lambda)^*$ is a left $H$-module map using the left regular $H$-actions on $H$, $H^*$:
\[
\begin{array}{rl}
h\cdot(\Psi_\Lambda)^*(\phi) &= h \cdot (\phi(\Lambda_1) \Lambda_2) = \phi(\Lambda_1)(h \Lambda_2) \smallskip\\
&\overset{\textnormal{$(\ast)$}}{=} \phi(S(h)\Lambda_1)\Lambda_2 = (h \cdot \phi)(\Lambda_1)\Lambda_2 = (\Psi_\Lambda)^*(h \cdot \phi),
\end{array}
\]
where $(\ast)$ holds by \cite[Lemma 3.2(b)]{BNS}. So, $\Psi_\Lambda$ is bijective if and only if there is a unique solution to the equation $\Psi_\Lambda(\lambda) = 1_H$, which then holds if and only if there is a unique solution to the equation $\Phi'_\Lambda(\lambda) = 1_H$.
\end{enumerate}
\end{remark}

\begin{theorem} \cite[Corollary~3.10, Theorems~3.11,~3.16,~3.18]{BNS} \label{thm:BNS} Let $H$ be a weak Hopf algebra. Then the following statements hold.
\begin{enumerate} [font=\upshape]
    \item $H$ is self-injective. \smallskip
    \item Non-zero left (and right) integrals of $H$ exist. \smallskip
    \item $H$ is Frobenius if and only if $H$ has a non-degenerate left integral. \smallskip
    \item If $H$ has a non-degenerate left integral $\Lambda$ so that $\Psi_\Lambda(\lambda) = 1_H$ for some $\lambda \in H^*$, then $\lambda$ is a non-degenerate left integral of $H^*$. \qed
\end{enumerate}
\end{theorem}

\subsection{Main result} \label{sec:weak-main} Now we present our main result on the non-counital Frobenius structure of finite-dimensional weak Hopf algebras.

\begin{theorem} \label{thm:main-weak} Let $H$ be a weak Hopf algebra. Then the following statements hold.
\begin{enumerate}[font=\upshape]
\item $H$ is non-counital Frobenius with a nonzero comultiplication map $\Delta$. \smallskip
\item  $\Delta$ is counital if and only if $H$ is Frobenius (e.g., if and only if $H$ has a non-degenerate left integral). In this case, the counit is a nondegenerate  left integral  of~$H^*$.
\end{enumerate}
\end{theorem}

\begin{proof}
(a) By Theorem~\ref{thm:BNS}(b), there exists a non-zero left integral $\Lambda$ of $H$. Moreover, by \cite[Lemma~3.2(a,b)]{BNS}, we have that $\Lambda_1 \otimes x \Lambda_2 = S(x) \Lambda_1 \otimes \Lambda_2$ for all $x \in H$. Apply $\id \otimes S$ to get that $\Lambda_1 \otimes S(\Lambda_2)S(x) = S(x) \Lambda_1 \otimes S(\Lambda_2)$. So, for all $h \in H$, we have that
\begin{equation} \label{eq:Lambda-S}
\Lambda_1 \otimes S(\Lambda_2)h \;=\; h \Lambda_1 \otimes S(\Lambda_2).
\end{equation}
Now by taking, 
\begin{equation} \label{eq:Delta-weak}
    \Delta(h) := \Lambda_1 \otimes S(\Lambda_2)h,
\end{equation}
for all $h \in H$, this part of the theorem holds by Lemma~\ref{lem:coassoc}.

\smallskip

(b)  Take the comultiplication map $\Delta = \Delta_\Lambda$, for $\Lambda \in I^L(H)$, as in \eqref{eq:Delta-weak}. For an element $\lambda \in H^*$, define \begin{equation} \label{eq:ep-weak}
    \varepsilon = \varepsilon_\lambda: H \to \kk, \quad h \mapsto \lambda(h).
\end{equation}
Now, $\Delta$ is counital via $\varepsilon$ if and only if 
\begin{align}
    (\varepsilon \otimes \id)\Delta(h) \; = \; \lambda(\Lambda_1)\;S(\Lambda_2) \; h \; = \; h  \;\;\; \forall h\; \in H,  \label{eq:ep-wk1} \\
    (\id \otimes \varepsilon)\Delta(h) \; \overset{\textnormal{\eqref{eq:Lambda-S}}}{=} \; h\;\Lambda_1 \; \lambda(S(\Lambda_2)) \; =\; h \;\;\; \forall h\; \in H .\label{eq:ep-wk2}
\end{align} 

Recall Remark~\ref{rem:nondeg}(a) for the definitions of the maps $\Phi_{\Lambda}, \Phi'_{\Lambda}: H^* \to H$. Then, 
$$\eqref{eq:ep-wk1} \;  \text{holds} \; \iff  \Phi_{\Lambda}(\lambda) = 1_H
\quad \quad \quad \text{and} \quad \quad \quad 
\eqref{eq:ep-wk2} \;  \text{holds} \; \iff  \Phi'_{\Lambda}(\lambda) = 1_H.$$
Therefore, $\varepsilon_\lambda$ is a counit for $\Delta_\Lambda$ if and only if there is a unique solution $\lambda$ to the equations $\Phi_\Lambda(\lambda) =1_H$ and $\Phi'_\Lambda(\lambda) =1_H$. 
But if $\Phi'_\Lambda(\lambda) =1_H$, then $\Phi_\Lambda(\lambda) = 1_H$. Indeed, for all $h \in H$:
\begin{align*}
\lambda\Big(\Phi_\Lambda(\lambda)\;h\Big) &= \lambda\Big(\lambda(\Lambda_1)S(\Lambda_2)h\Big)  = \lambda(\Lambda_1)\lambda\Big(S(\Lambda_2)h\Big) = \lambda\Big(\Lambda_1 \lambda\big(S(\Lambda_2)h\big)\Big)\\ &\overset{\textnormal{\eqref{eq:Lambda-S}}}{=} \lambda\Big(h\Lambda_1 \lambda\big(S(\Lambda_2)\big)\Big)
 = \lambda\Big(h\;\Phi'_\Lambda(\lambda)\Big) = \lambda\Big(1_H \; h\Big).
 \end{align*}
Considering the right regular action $\triangleleft$ of $H$ on $H^*$, we then get that $\lambda \triangleleft \Phi_\Lambda(\lambda) = \lambda \triangleleft 1_H$. Thus,  $\Phi_\Lambda(\lambda) = 1_H$ since the action $\triangleleft$ is faithful. Now by Remark~\ref{rem:nondeg}(b), $\varepsilon_\lambda$ is a counit if and only if the left integral $\Lambda$ is non-degenerate. The last statement follows from Remark~\ref{rem:nondeg}(b) and  Theorem~\ref{thm:BNS}(d).
\end{proof}


\section{Examples for finite dimensional weak Hopf algebras }\label{sec:weak-ex}

In this part, we provide examples of the main result of Section~\ref{sec:weak}, Theorem~\ref{thm:main-weak}, on the non-counital Frobenius condition for finite-dimensional weak Hopf algebras. In Section~\ref{sec:groupoid}, we illustrate how groupoid algebras are (counital) Frobenius. Moreover, in Section~\ref{sec:QTG}, we show that certain weak Hopf algebras, called  quantum transformation groupoids, are (counital) Frobenius.  In both cases, we construct an explicit non-degenerate left integral of the weak Hopf algebra $H$ under investigation, and derive formulas for the comultiplication and counit that makes $H$ (counital) Frobenius.

\subsection{Groupoid algebras} \label{sec:groupoid}
Take $\mathcal{G}$ to be a {\it finite groupoid}, that is,  a category with finitely many objects $\mathcal{G}_0$, and finitely many morphisms $\mathcal{G}_1$ which are all isomorphisms. For $g \in \mathcal{G}_1$, let $s(g)$ and $t(g)$ denote the source and target of $g$, respectively.

\begin{definition} Given a finite groupoid $\mathcal{G}$, a {\it groupoid algebra} $\kk \mathcal{G}$ is a finite-dimensional weak Hopf algebra, which is spanned by $g \in \mathcal{G}_1$ as a $\kk$-vector space, with product $g h$ being the composition $g\circ h$ if $g$ and $h$ are composable and $0$ otherwise, and with unit $\sum_{X \in \mathcal{G}_0} \id_X$. Moreover, for $g \in \mathcal{G}_1$, we have that  $\Delta_{\textnormal{wk}}(g)=g\otimes g$, $\varepsilon_{\textnormal{wk}}(g)=1_\kk$, and $S(g)=g^{-1}$.
\end{definition}
Now consider the next result.

\begin{proposition}
The groupoid algebra $\kk \mathcal{G}$ is Frobenius.
\end{proposition}

\begin{proof}
We will show that $\kk \mathcal{G}$ has a non-degenerate left integral. From \cite[Example~3.1.2]{nikshych2002finite}, recall that
    $I^L(\kk \mathcal{G}) = $ span$\{\textstyle \sum_{h \in \mathcal{G}_1 : t(h) = X} h\}_{X \in \mathcal{G}_0}.$ Consider $$\Lambda = \textstyle \sum_{h\in \mathcal{G}_1}h \; \; \in I^L(\kk \mathcal{G}),$$
      and the linear map 
      \begin{center}
          $\lambda:\kk \mathcal{G} \rightarrow \kk$, \;\;  where $\lambda(g)=1_\kk $ if $g=\id_X$ for some $X \in \mathcal{G}_0$, \; \; and $\lambda(g) = 0$ otherwise.
      \end{center} 
      Then, observe that $\lambda (\Lambda_1)  \Lambda_2=1_{\kk \mathcal{G}}$.
Hence, $\Lambda$ is non-degenerate, and hence $\lambda$ is a non-degenerate integral of $(\kk \mathcal{G})^*$ by Theorem~\ref{thm:BNS}(d). Now by Theorem~\ref{thm:BNS}(c), $\kk \mathcal{G}$ is Frobenius.

Moreover, we can use the non-degenerate integrals above, along with Theorem~\ref{thm:main-weak}, to see that $\kk \mathcal{G}$ is Frobenius via Definition-Theorem~\ref{defthm:Frob}. Here, the comultiplication $\Delta = \Delta_\Lambda$ is given by \eqref{eq:Delta-weak}:
$$\Delta(g)= \textstyle \sum_{h\in \mathcal{G}_1} h\otimes (h^{-1} g)$$ for $g \in \mathcal{G}_1$,
and with counit $\varepsilon = \lambda$ as in \eqref{eq:ep-weak}.
\end{proof}

\subsection{Quantum transformation groupoids}
\label{sec:QTG} Here, we discuss the Frobenius condition for certain weak Hopf algebras, called quantum transformation groupoids, which are constructed using the data of a Hopf algebra $L$ and a strongly separable module algebra $B$ over $L$.

\begin{notation} Consider the following notation.
\begin{itemize}
    \item Take $L$ to be finite-dimensional Hopf algebra over $\kk$, with comultiplication $\Delta_L$, counit $\varepsilon_L$, and antipode $S_L$ [Definition-Proposition~\ref{def:Hopf}].
    
    \smallskip
     
    
    \smallskip
    
    \item Let $B$ be a strongly separable algebra over $\kk$, which implies that it comes equipped with an element  $e^1\otimes e^2 \in B \otimes B$ satisfying
    \begin{align}
        be^1\otimes e^2 & = e^1\otimes e^2b \quad \quad  \forall b \in B, \label{eq:idempotent1} \\
        e^1e^2 & = 1_B, \label{eq:idempotent2} \\
        e^1\otimes e^2 & = e^2\otimes e^1. \label{eq:idempotent3}
    \end{align}
    Here, $e^1 \otimes e^2$ is called a {\it symmetric separability idempotent}. 
    
\smallskip

    \item Furthermore, there exists a non-degenerate trace form $\omega: B\rightarrow \kk$ defined by 
    \begin{align}
        \omega(e^1)e^2 = 1_B = e^1\omega(e^2) \label{eq:trace}.
    \end{align}
    
    \smallskip
    
    \item We further impose that $B$ is a {\it right $L$-module algebra} via $\triangleleft$, that is, we have a map $\triangleleft: B\otimes L\rightarrow B$ satisfying 
    \begin{align}
        (b\triangleleft \ell)\triangleleft \ell' = b\triangleleft (\ell \ell'), & \hspace{1cm}
        b\triangleleft 1_L=b, \label{eq:QTGaction1}\\
        (bb')\triangleleft \ell = (b\triangleleft \ell_1)(b'\triangleleft \ell_2), & \hspace{1cm}
        1_B\triangleleft \ell=\varepsilon_L(\ell)1_B
        \label{eq:QTGaction2}
    \end{align}
    for all $b \in B$ and $\ell, \ell' \in L$. 
    
     \smallskip
     
    \item Moreover, we assume that the separability idempotent satisfies the identity below:
\begin{align}\label{eq:idempotentAction}
       (e^1\triangleleft \ell) \otimes e^2 = e^1\otimes (e^2\triangleleft S_L(\ell))
\end{align}
for all $\ell \in L$
\end{itemize}
\end{notation}

\begin{definition}
Recall the notation above. A {\it quantum transformation groupoid}   is a weak Hopf algebra over $\kk$, which as a $\kk$-vector space is $$H:= H(L,B)=B^{op}\otimes L\otimes B,$$ with the following structure maps:
\begin{align}
\text{multiplication: } &   (a\otimes \ell \otimes b)(a'\otimes \ell'\otimes b')  = (a'\triangleleft S_L(\ell_1))a \otimes \ell_2 \ell'_1 \otimes (b\triangleleft \ell'_2)b'; \label{eq:QTGmult} 
\\
\text{unit: } &  1_B \otimes 1_L \otimes 1_B; \label{eq:QTGunit} 
\\
\text{comultiplication: } & \Delta_{\text{wk}}(a\otimes \ell\otimes b)  = (a\otimes \ell_1 \otimes e^1) \otimes ((e^2 \triangleleft S_L(\ell_2)) \otimes \ell_3 \otimes b);
\label{eq:QTGcomult} 
\\
\text{counit: } &  \varepsilon_{\text{wk}}(a\otimes \ell\otimes b) = \omega(a(b \triangleleft S_L^{-1}(\ell))); \label{eq:QTGcounit} 
\\
\text{antipode: } & S_{\text{wk}}(a\otimes \ell\otimes b)  = b\otimes S_L(\ell)\otimes a. \label{eq:QTGantipode}
\end{align}
\end{definition}

We refer the reader to \cite[Section~7]{WWW} and references therein for more details about the structure of the weak Hopf algebras $H(L,B)$.

\smallskip 

Next, recall some facts about integrals of finite-dimensional Hopf algebras. 

\begin{definition} \cite[Definition~10.1.1]{radford} Recall the notation above.
A {\it left} (resp., {\it right}) {\it integral} of $L$ is an element $\Lambda \in L$ such that $\ell \Lambda = \varepsilon_L(\ell) \Lambda$ (resp., $\Lambda \ell = \varepsilon_L(\ell) \Lambda$) for all $\ell \in L$.
\end{definition}

The notion above is consistent with Definition~\ref{def:integral} via Definition-Proposition~\ref{def:Hopf}. Moreover, we have the following facts.

\begin{proposition} \cite[Proposition 10.1.3(b), Section~10.2]{radford} \label{prop:L-integ}
 Recall the notation above. 
\begin{enumerate} [font=\upshape]
\item Non-zero left (or, right) integrals of finite-dimensional Hopf algebras $L$ exist and are always non-degenerate (and thus, finite-dimensional Hopf algebras are always Frobenius). 
\smallskip
    \item If $\Lambda$ is a right integral of $L$, then
\begin{align} \label{eq:rightintegL}
\ell S(\Lambda_1) \otimes \Lambda_2 =  S(\Lambda_1) \otimes \Lambda_2 \ell,
\end{align}
for all $\ell \in L$.
\end{enumerate}
\vspace{-.2in}

\qed
\end{proposition}

\begin{proposition} \label{prop:QTG-nondeg}
Recall the notation above. If $\Lambda$ is a  right integral for the finite-dimensional Hopf algebra $L$, then 
\begin{equation} \label{eq:barLambda} \bar{\Lambda}:=(e^1\triangleleft \Lambda_1) \otimes S_L(\Lambda_2) \otimes e^2
\end{equation}
is a non-degenerate left integral of the quantum transformation groupoid $H(L,B)$.  
\end{proposition}

\begin{proof}
First, we show that if $\Lambda$ is a right integral of $L$, then the element 
$$(e^1 \triangleleft S_L(\Lambda_1)) \otimes \Lambda_2 \otimes e^2 $$
is a right integral of $H(L,B)$. We compute:
{\small
\begin{align*}
((e^1 \triangleleft S_L(\Lambda_1)) \otimes \Lambda_2 \otimes e^2)(a' \otimes \ell' \otimes b') 
&\overset{\textnormal{(\ref{eq:QTGmult})}}{=} (a'\triangleleft S_L(\Lambda_2)) (e^1\triangleleft S_L(\Lambda_1)) \otimes \Lambda_3 \ell'_1 \otimes (e^2\triangleleft \ell'_2)b'  
 \medskip \\ 
  & \overset{\textnormal{(*)}}{=}  (a'\triangleleft S_L(\Lambda_1)_1) (e^{1}\triangleleft S_L(\Lambda_1)_2) \otimes \Lambda_2 \ell'_1 \otimes (e^2\triangleleft \ell'_2)b'
 \medskip \\
  & \overset{\textnormal{(\ref{eq:QTGaction2})}}{=}  (a'e^1 \triangleleft S_L(\Lambda_1))  \otimes \Lambda_2 \ell'_1 \otimes (e^2\triangleleft \ell'_2)b'
 \medskip \\
  & \overset{\textnormal{(\ref{eq:idempotent1})}}{=}  (e^1 \triangleleft S_L(\Lambda_1))  \otimes \Lambda_2 \ell'_1 \otimes ((e^2 a')\triangleleft \ell'_2)b'  
 \medskip \\
  & \overset{\textnormal{(\ref{eq:QTGaction2})}}{=}  (e^1 \triangleleft S_L(\Lambda_1))  \otimes \Lambda_2 \ell'_1 \otimes (e^2\triangleleft \ell'_2) (a'\triangleleft \ell'_3) b' 
 \medskip \\
  & \overset{\textnormal{\eqref{eq:rightintegL}}}{=}  (e^1 \triangleleft \ell'_1 S_L(\Lambda_1))  \otimes \Lambda_2 \otimes (e^2\triangleleft \ell'_2) (a\triangleleft \ell'_3) b' 
 \medskip \\
  & \overset{\textnormal{(\ref{eq:QTGaction1})}}{=}  ((e^1 \triangleleft \ell'_1)\triangleleft S_L(\Lambda_1))  \otimes \Lambda_2 \otimes (e^2\triangleleft \ell'_2) (a\triangleleft \ell'_3) b' 
 \medskip \\
  & \overset{\textnormal{(\ref{eq:idempotentAction})}}{=}  (e^1 \triangleleft S_L(\Lambda_1))  \otimes \Lambda_2 \otimes ((e^2 \triangleleft S_L(\ell'_1) ) \triangleleft \ell'_2) (a\triangleleft \ell'_3) b' 
 \medskip \\
  & \overset{ \textnormal{(\ref{eq:QTGaction1})}}{=}   (e^1 \triangleleft S_L(\Lambda_1))  \otimes \Lambda_2 \otimes (e^2 \triangleleft (S_L(\ell'_1) \ell'_2)) (a\triangleleft \ell'_3) b'
 \medskip \\
  & \overset{\textnormal{(**)}}{=}  (e^1 \triangleleft S_L(\Lambda_1))  \otimes \Lambda_2 \otimes (e^2 \triangleleft \varepsilon_L(\ell'_1)1_L) (a\triangleleft \ell'_2) b'
 \medskip \\
  & \overset{\textnormal{(\ref{eq:QTGaction1}),(\%)}}{=}  (e^1 \triangleleft S_L(\Lambda_1))  \otimes \Lambda_2 \otimes e^2 (a\triangleleft \ell') b' 
  \medskip \\
  & \overset{\textnormal{(\ref{eq:QTGmult})}}{=}   ((e^1 \triangleleft S_L(\Lambda_1)) \otimes \Lambda_2 \otimes e^2) (1_B \otimes 1_L \otimes (a'\triangleleft \ell')b')
 \medskip \\
  &\overset{\textnormal{($\star$)}}{=}   ((e^1 \triangleleft S_L(\Lambda_1)) \otimes \Lambda_2 \otimes e^2)\; \varepsilon_s(a' \otimes \ell' \otimes b')
\end{align*}
}

\noindent for $a',b'\in B$ and $\ell' \in L$. Here, $(*)$ is anti-comultiplicativity of the antipode,   $(**)$ is the antipode axiom, $(\%)$ is the counit axiom, and $(\star)$ follows from \cite[Lemma~7.9]{WWW}.

\smallskip

Next,  by \cite[Lemma~2.9]{BNS}, we have that
\begin{align*}
    \bar{\Lambda}:= \;S_{\text{wk}}((e^1 \triangleleft S_L(\Lambda_1)) \otimes \Lambda_2 \otimes e^2) &= e^2 \otimes S_L(\Lambda_2) \otimes (e^1 \triangleleft S_L(\Lambda_1)) \medskip \\
    &\overset{\textnormal{\eqref{eq:idempotent3}}}{=} e^1 \otimes S_L(\Lambda_2) \otimes (e^2 \triangleleft S_L(\Lambda_1)) \medskip \\
    &\overset{\textnormal{\eqref{eq:idempotentAction}}}{=} (e^1 \triangleleft \Lambda_1) \otimes S_L(\Lambda_2) \otimes e^2 
    \end{align*}
is a left integral of $H(L,B)$.

\smallskip

Finally, we will verify the non-degeneracy condition for the left integral $\bar{\Lambda}$ of $H(L,B)$. By Remark~\ref{rem:nondeg}(b),  it suffices to show that there exists an element $\bar{\lambda} \in H^*$ such that \begin{align} \label{eq:Phi-bar-lam}
    \Psi_{\bar{\Lambda}}(\bar{\lambda}) \; = \;\bar{\Lambda}_1 \; \bar{\lambda}(\bar{\Lambda}_2)\; = \;1_H.
\end{align}
Since, $S_L(\Lambda)$ is left integral of $L$,  we can choose an element $\lambda$ of $L$ so that 
\begin{align} \label{eq:lam-S-1L}
    S_L(\Lambda)_1\;\lambda(S_L(\Lambda)_2) = \lambda(S_L(\Lambda_1)) \;S_L(\Lambda_2) = 1_L
\end{align}
by Proposition~\ref{prop:L-integ}(a) and Remark~\ref{rem:nondeg}(b).  We claim that
\begin{align}\label{eq:barlambda}
    \bar{\lambda} = \omega \otimes \lambda \otimes \omega
\end{align}
is the desired element that makes \eqref{eq:Phi-bar-lam} hold. We take $e'^1 \otimes e'^2$ to be a copy of $e^1 \otimes e^2$ in the computations below:
{\small
\begin{align*}
        \bar{\Lambda}_1 \; \bar{\lambda}(\bar{\Lambda}_2)  
        &\overset{\textnormal{(\ref{eq:QTGcomult})}}{=}  
        [ (e^1\triangleleft\Lambda_1) \otimes S_L(\Lambda_2)_1 \otimes e'^1 ] \; \; \bar{\lambda}[(e'^2\triangleleft S_L(S_L(\Lambda_2)_2)) \otimes S_L(\Lambda_2)_3 \otimes e^2]
        \medskip \\
        &\overset{\textnormal{(\ref{eq:idempotentAction})}}{=} 
        [(e^1\triangleleft \Lambda_1) \otimes S_L(\Lambda_2)_1 \otimes (e'^1 \triangleleft S_L(\Lambda_2)_2) ] \; \; \bar{\lambda}[e'^2 \otimes S_L(\Lambda_2)_3 \otimes e^2]
        \medskip \\
        &
        \overset{\textnormal{(\ref{eq:barlambda})}}{=} 
        [ (e^1\triangleleft \Lambda_1) \otimes S_L(\Lambda_2)_1 \otimes (e'^1 \triangleleft S_L(\Lambda_2)_2) ] \; \; \omega(e'^2) \; \lambda(S_L(\Lambda_2)_3) \; \omega(e^2 )
        \medskip \\
        &
        =
        \lambda(S_L(\Lambda_2)_3)\;\;  [  ((e^1 \;\omega(e^2)) \triangleleft \Lambda_1) \otimes S_L(\Lambda_2)_1 \otimes ((e'^1 \; \omega(e'^2)) \triangleleft S_L(\Lambda_2)_2) ]    
        \medskip \\
        &
        \overset{\textnormal{(\ref{eq:trace})}}{=}  
        \lambda(S_L(\Lambda_2)_3)\;\;  [(1_B  \triangleleft \Lambda_1) \otimes S_L(\Lambda_2)_1 \otimes (1_B \triangleleft S_L(\Lambda_2)_2)]  
        \medskip \\
        &
        \overset{\textnormal{\eqref{eq:QTGaction2}}}{=} 
        \lambda(S_L(\Lambda_2)_3)\;\;  [  \varepsilon_L(\Lambda_1)1_B \otimes S_L(\Lambda_2)_1 \otimes \varepsilon_L(S_L(\Lambda_2)_2)1_B ] 
        \medskip \\
        &
        \overset{\textnormal{(*)}}{=}  
        \lambda(S_L(\Lambda_2))\;\;  [  \varepsilon_L(\Lambda_1) 1_B \otimes S_L(\Lambda_4) \otimes \varepsilon( S_L(\Lambda_3)) 1_B ]
        \medskip \\
        &
        \overset{\textnormal{(\%)}}{=} 
        \lambda(S_L(\Lambda_1)) \; [  1_B \otimes S_L(\Lambda_3) \otimes \varepsilon_L( S_L(\Lambda_2))1_B ]
        \medskip \\
         &
        \overset{\textnormal{$\varepsilon_L S_L \hspace{-.03in}=\hspace{-.03in} S_L$}}{=} 
        \lambda(S_L(\Lambda_1)) \; [  1_B \otimes S_L(\Lambda_3) \otimes \varepsilon_L(\Lambda_2)1_B ]
        \medskip \\
          &
        \overset{\textnormal{(\%)}}{=} 
         \;  1_B \otimes \lambda(S_L(\Lambda_1)) S_L(\Lambda_2) \otimes 1_B 
        \medskip \\
        &
        \overset{\textnormal{\eqref{eq:lam-S-1L}}}{=} 
        1_B \otimes 1_L \otimes 1_B \medskip \\
        &\overset{\textnormal{\eqref{eq:QTGunit}}}{=} 
        1_{H(L,B)}.
\end{align*}
}

\noindent Here, $(*)$ is anti-comultiplicativity of the antipode, and (\%) is the counit axiom.
 Therefore, $\bar{\Lambda}$ is a non-degenerate left integral of $H(L,B)$ as claimed.
\end{proof}

\begin{corollary}
The quantum transformation $H(L,B)$ is a Frobenius algebra via maps $\Delta_{\bar{\Lambda}},\; \varepsilon_{\bar{\lambda}}$ defined as follows:
\begin{align*}
    \Delta_{\bar{\Lambda}}(a\otimes \ell \otimes b) 
    &  =
    [(e^1\triangleleft \Lambda_1 S_L(\ell_1)) a \otimes 
    \ell_2 S_L(\Lambda_4) \otimes (be'^1 \triangleleft S_L(\Lambda_3)) ] 
    \otimes 
    [e^2 \otimes S^2_L(\Lambda_2) \otimes e'^2], \medskip\\
    \varepsilon_{\bar{\lambda}}(a\otimes \ell \otimes b) & = \omega(a)\; \lambda(\ell)\; \omega(b),
\end{align*}
for $a,b \in B$ and $\ell \in L$.
Here, $\Lambda$ is a right integral of $L$, and  $e'^1 \otimes e'^2$ is a copy of the separability idempotent $e^1 \otimes e^2$ of $B$. Moreover, $\lambda$ is a choice of element of $L^*$ such that~\eqref{eq:lam-S-1L} holds; in fact, $\lambda$ is a non-degenerate left integral of $L^*$.
\qed
\end{corollary}

\begin{proof}
The fact that $H:=H(L,B)$ is  Frobenius follows from Proposition~\ref{prop:QTG-nondeg} and Theorem~\ref{thm:BNS}(c). The formulas for the comultiplication and counit maps for the Frobenius structure of $H$ then follow from the formulas for the non-degenerate integrals $\bar{\Lambda}$ and $\bar{\lambda}$ of $H$ and of $H^*$, resp., in \eqref{eq:barLambda} and \eqref{eq:barlambda}. Namely, by \eqref{eq:QTGcomult} and \eqref{eq:idempotentAction}, we have that
\begin{equation} \label{eq:Delta-bar}
    \Delta_{\text{wk}}(\bar{\Lambda}) = [(e^1\triangleleft \Lambda_1) \otimes S_L(\Lambda_4) \otimes (e'^1 \triangleleft S_L(\Lambda_3)) ] \; \otimes \; [e'^2 \otimes S_L(\Lambda_2) \otimes e^2].
    \end{equation}
Then, to get $\Delta$ and $\varepsilon$, one needs to apply the formulas \eqref{eq:Delta-weak} and \eqref{eq:ep-weak}, resp., in the proof of  Theorem~\ref{thm:main-weak}. In particular, we have that:
{\small
\begin{align*}
  &\Delta_{\bar{\Lambda}}(a\otimes \ell \otimes b)
     =
    [(a\otimes \ell \otimes b)\; \bar{\Lambda}_1] \otimes S_{\text{wk}}(\bar{\Lambda}_2) 
    \\ \medskip
    &  \overset{\textnormal{\eqref{eq:Delta-bar},\eqref{eq:QTGmult}}}{=} 
    [(e^1\triangleleft \Lambda_1 S_L(\ell_1)) a \otimes 
    \ell_2 S_L(\Lambda_5) \otimes (b \triangleleft S_L(\Lambda_4))(e'^1 \triangleleft S_L(\Lambda_3)) ] 
    \otimes 
    S_{\text{wk}}[e'^2 \otimes S_L(\Lambda_2) \otimes e^2] \\ \medskip
    &  \overset{\textnormal{\eqref{eq:QTGaction2}}}{=} 
    [(e^1\triangleleft \Lambda_1 S_L(\ell_1)) a \otimes 
    \ell_2 S_L(\Lambda_4) \otimes (be'^1 \triangleleft S_L(\Lambda_3)) ] 
    \otimes 
    S_{\text{wk}}[e'^2 \otimes S_L(\Lambda_2) \otimes e^2]
    \\ \medskip
    &  \overset{\textnormal{\eqref{eq:QTGantipode}}}{=} 
    [(e^1\triangleleft \Lambda_1 S_L(\ell_1)) a \otimes 
    \ell_2 S_L(\Lambda_4) \otimes (be'^1 \triangleleft S_L(\Lambda_3)) ] 
    \otimes 
    [e^2 \otimes S^2_L(\Lambda_2) \otimes e'^2].
    \end{align*}}

\noindent Finally, the last statement holds by  Theorem~\ref{thm:BNS}(d).
\end{proof}

\begin{example}
If we take $L= \kk$, then $\Lambda = \lambda = 1_\kk$, and we have the following structure formulas for the Frobenius weak Hopf algebra $H:=H(\kk,B) = B^{op} \otimes B$:
\begin{itemize}
    \item algebra: $m((a \otimes b)\otimes (a' \otimes b)):=(a \otimes b)(a' \otimes b) = a'a \otimes bb'$, \quad  $1_H = 1_B \otimes 1_B$;\smallskip
    \item weak Hopf: $\Delta_{\text{wk}}(a \otimes b) = (a \otimes e^1) \otimes (e^2 \otimes b)$, \; $\varepsilon_{\text{wk}}(a \otimes b) = \omega(ab)$, \; $S_{\text{wk}}(a \otimes b) = b \otimes a$;\smallskip
    \item Frobenius: $\Delta(a \otimes b) = (e^1 a \otimes be'^1) \otimes (e^2 \otimes e'^2)$, \; $\varepsilon(a \otimes b) = \omega(a) \omega(b)$;
\end{itemize}
for $a,b \in B$. Indeed, let us check that $(H,\Delta, \varepsilon)$ is a coassociative, counital coalgebra:
{\small
\begin{align*}
(\Delta \otimes \id)\Delta(a \otimes b) &= \Delta(e^1 a \otimes be'^1) \otimes (e^2 \otimes e'^2)\\
&= (e''^1 e^1 a \otimes b e'^1 e'''^1) \otimes (e''^2 \otimes e'''^2) \otimes (e^2 \otimes e'^2)
\\
&= (e^1 e''^1 a \otimes b e'''^1 e'^1) \otimes (e^2 \otimes e'^2) \otimes (e''^2 \otimes e'''^2)\\
& \overset{\textnormal{\eqref{eq:idempotent1},\eqref{eq:idempotent3}}}{=}
  (e^1 a \otimes be'^1) \otimes (e''^1 e^2 \otimes e'^2 e'''^1) \otimes (e''^2 \otimes e'''^2)\\
  &=(e^1 a \otimes be'^1) \otimes \Delta(e^2 \otimes e'^2) \quad = (\id \otimes \Delta)\Delta(a \otimes b);
\end{align*}
}

\vspace{-.2in}

{\small
\begin{align*}
   (\varepsilon \otimes \id)\Delta(a \otimes b) 
   &= \omega(e^1 a) \;\omega(be'^1)\;(e^2 \otimes e'^2) \overset{\textnormal{\eqref{eq:idempotent1},\eqref{eq:idempotent3}}}{=} \omega(e^1) \;\omega(e'^1)\;(a e^2 \otimes e'^2b) \overset{\textnormal{\eqref{eq:trace}}}{=} a\otimes b;\\
      (\id \otimes \varepsilon)\Delta(a \otimes b)
   &= (e^1 a \otimes be'^1) \; \omega(e^2) \; \omega(e'^2) \overset{\textnormal{\eqref{eq:trace}}}{=} a\otimes b.
\end{align*}
}

\noindent Lastly, \eqref{eq:Delta-m} holds by the following computations: 
{\small
\begin{align*}
    \Delta((a \otimes b)(a'\otimes b')) &= (e^1 a'a \otimes bb'e'^1) \otimes (e^2 \otimes e'^2) =:(\star);\\
    (m \otimes \id)(\id \otimes \Delta)((a \otimes b) \otimes (a' \otimes b')) &= (a \otimes b) (e^1 a' \otimes b'e'^1) \otimes (e^2 \otimes e'^2) = (\star);
    \\
    (\id \otimes m)(\Delta \otimes \id)((a \otimes b) \otimes (a' \otimes b')) &=  (e^1 a \otimes be'^1) \otimes (e^2 \otimes e'^2)(a' \otimes b') \overset{\textnormal{\eqref{eq:idempotent1}}}{=} (\star).
    \end{align*}
    }
\end{example}


\section*{Acknowledgements}
The authors would like to thank Alex Chirvasitu for insightful discussions on the first version of this manuscript, and thank Robert Lipshitz for informing us of  the reference \cite{CohenGodin} in connection to our work. The authors would also like to thank the anonymous referee for providing comments that greatly improved the exposition of our manuscript. The first and second authors were partially supported by the US National Science Foundation grant DMS-2100756.


\bibliography{NonCounitalFrob}
\bibliographystyle{alpha}

\end{document}